\documentclass[12pt, a4paper]{article}
\usepackage{amssymb}
\usepackage{stmaryrd}
\usepackage{enumitem}

\usepackage{geometry}
\usepackage{tikz}
\usetikzlibrary{shapes,positioning}
\usepackage{float}
\usepackage{url}
\usepackage{mathtools}
\usepackage{amsmath,amsfonts,amsthm}
\usepackage{mathrsfs}

\newcommand{\problem}[1]{\paragraph*{Problem #1:}}

\newtheorem{theorem}{Theorem}[section]
\newtheorem{lemma}[theorem]{Lemma}
\newtheorem{proposition}{Proposition}[section]

\newtheorem{definition}{Definition}

\usepackage{subfig}

\usepackage{ytableau}
\usepackage[colorlinks,
linkcolor=blue,
anchorcolor=blue,
citecolor=blue
]{hyperref}
\begin{document}
\begin{center}
Regularities of Typical Sigma Index on Caterpillar Trees of Pendent Vertices\\[6pt]
 Jasem Hamoud$^{1}$ \hspace{0.5cm} Duaa Abdullah$^{2}$\\[6pt]
$^{1,2}$ Discrete Mathematics, Moscow Institute of Physics and Technology (National Research University)\\
 Email: $^{1}${\tt hamoud.math@gmail.com}, \hspace{0.3cm} $^{2}${\tt duaa1992abdullah@gmail.com}
\end{center}
\noindent
\begin{abstract}
In this paper, topological indices play a significant role in the analysis of caterpillar trees, especially due to their applications in chemical graph theory. We presented a study on topological indices related to the Sigma index, which we carefully selected on caterpillar trees with multiple levels. Through the election of these topological indices, we provide efficient models of these indices on caterpillar trees, as it is known that the Albertson index is the basis on which most of the topological indices are built and we have shown this through the close correlation between the Albertson's index and the Sigma index.
\end{abstract}

\noindent\textbf{AMS Classification 2010:} 05C05, 05C12, 05C20, 05C25, 05C35, 05C76, 68R10.


\section{Introduction}

Throughout this paper, we consider finite, undirected, and simple graphs. Caterpillar trees, a type of graph in graph theory, were initially studied in a number of works by F.~Harary and E.~M.~Schwenk~\cite{jas1}. 
Basically, a caterpillar tree is one in which all vertices lie within one edge of a central path. A tree in which all vertices are either directly connected to or are part of a central path \emph{stalk} is called a \emph{caterpillar graph}.

The \emph{Albertson} and \emph{Sigma} indices are useful tools for analyzing caterpillar trees in graph theory. 
Their study not only advances theory but also has practical implications in various academic fields. 
Denote by $\Delta$ and $\delta$ the graph's maximum and minimum degrees, respectively (see Lemma~\ref{lem1}). Horoldagva B.et al.~\cite{32a} introduced, for integers $n, \operatorname{irr} \geq 0$ with $1 < \operatorname{irr} < n$, where $\operatorname{irr}$ is the \emph{irregularity index} (see~\eqref{intro:eq1}), that if $G$ is a maximally irregular graph of order $n$ and $|E(G)| \leq \left\lfloor \frac{n^2}{4} \right\rfloor$. Then
\begin{equation}~\label{intro:eq1}
|E(G)| \leq \left\lfloor \frac{\operatorname{irr}(2n - \operatorname{irr} + 1)}{4} \right\rfloor.
\end{equation}
In \cite{32a}, the ``Collatz and Sinogowitz'' index is introduced as $CS(G) = \lambda_1 - \frac{2m}{n}$, where $\lambda_1$ is the largest eigenvalue of the adjacency matrix. 
The \emph{variance} of vertex degrees is defined as
\begin{equation}~\label{intro:eq2}
\operatorname{Var}(G) = \frac{1}{n} \sum_{v \in V(G)} d_G(v)^2 - \left(\frac{1}{n} \sum_{v \in V(G)} d_G(v)\right)^2.
\end{equation}

The \emph{Albertson index} was defined in 1997 as the \emph{irregularity} of $G$ given by the sum of imbalances of all edges:
\[
\operatorname{irr}(G) = \sum_{uv \in E(G)} |d_G(u) - d_G(v)|.
\]

The general Albertson irregularity index of a graph was defined in \cite{57a} (also see~\cite{59a}) by Z.~Lina, et al., such that for any $p > 0$,
\begin{equation}~\label{intro:eq3}
\operatorname{irr}_p(G) = \left( \sum_{uv \in E(G)} |d(u) - d(v)|^p \right)^{\frac{1}{p}}.
\end{equation}

In \cite{7a}, it was proved that
\[
\operatorname{irr}_T(G) \leq \frac{n^2}{4} \operatorname{irr}(G).
\]
If $G$ is a tree, then
\[
\operatorname{irr}_T(G) \leq (n-2) \operatorname{irr}(G).
\]
G.~H.~Fath, et al.~\cite{15a} mention that for all trees of order $n$ with $n > 2$, we have $\operatorname{irr}(T) \geq 2$, and for the star $S_n$, $\operatorname{irr}(S_n) = (n-1)(n-2)$. Thus, the authors determine the maximum value of the irregularity index as $\operatorname{irr}_{\max} = n - 2$. S.~Dorjsembe, et al.~\cite{36a} consider a path $u_0 u_1 \cdots u_t$ in $G$, and define
\begin{equation}~\label{intro:eq4}
imb_G(u_0, u_t) = \sum_{i=0}^{t-1} |d_G(u_i) - d_G(u_{i+1})|.
\end{equation}
They further state that according to~\eqref{intro:eq1}--\eqref{intro:eq4},
\[
\operatorname{irr}(G) \geq 2 \left\lfloor \frac{\Delta}{2} \right\rfloor,
\]
and provide the upper bound:
\[
\operatorname{irr}(G) \leq \frac{1}{48} \left(6 n^2 \Delta + 3 \Delta^2 n - 2 \Delta^3 - 4 \Delta\right).
\]
If the tree $T$ is a spider with $n$ vertices, then $\operatorname{irr}(T) = \Delta(\Delta - 1).$ But this relationship is not true on all graphs, as through the study we show that it does not fit all trees. Furthermore, see~\cite{HamoudwithDuaa, HamoudwithDuaaPn2T}. 
Therefore, in \cite{41a} the maximum degree $\Delta$ of stepwise irregular graphs of order $n$ satisfies
\[
\Delta \leq \left\lfloor \frac{n+1}{2} \right\rfloor.
\]
Thus, 
\[
\begin{cases}
\Delta \leq \frac{n+9}{4} & \text{ if }\,\,   \delta \leq \Delta - 4, \\
\Delta < \frac{n+6}{4} & \text{ if }\,\, \delta > \Delta - 4, \\
\Delta < \frac{n+2}{4} & \text{ if }\,\,  \Delta < \frac{n+6}{4}.
\end{cases}
\]

The relevance of this paper stems from the fact that both indices are used in extremal graph theory, which analyzes the maximum and minimum values of specific graph properties under certain constraints. These indices help find extremal structures that optimize or limit irregularity in caterpillar trees, which are distinguished by a core ``backbone'' with leaves attached. Within biological graph theory, topological indices like the Albertson and Sigma indices are used to predict molecular properties based on structural representations. Caterpillar trees can model specific molecular frameworks, making them valuable for anticipating the chemical and physical features of materials.

The Randi\'c index ($R(G)$) is an important topological index in graph theory, with applications in chemistry and pharmacology. The Randić index is a significant topological invariant in graph theory, commonly used to measure structural properties of molecular graphs in chemistry as well as network connectivity and robustness in mathematics and computer science. In 1975, Milan Randi\'c introduced a method for measuring the degree of branching in molecular structures, it is define by Li, X., \& Shi, Y.in~\cite{Li2008shi} as: 
\begin{equation}~\label{eq1rand}
R(G)=\sum_{uv\in E(G)} \frac{1}{\sqrt{d_ud_v}}.
\end{equation}
and the general Randi\'c index with $\alpha$ an arbitrary real number given as:
\begin{equation}~\label{eq2rand}
R_\alpha(G)=\sum_{uv\in E(G)} (d_ud_v)^\alpha.
\end{equation}

The aim of this paper is to present a study on the effect of topological indices on caterpillar trees, as these trees lead to other diverse trees. Given the great importance of studying the impact of topological indices, especially in chemistry and computer science, it was necessary to emphasize this significance and present this work, through which we clarify this effect and strive to link the topological indices introduced in this paper.

%==================================
 \section{Preliminaries}\label{sec2}
%==================================
In Lemma~\ref{lem1}, we show the relationship between the maximum and minimum degree according to the Albertson index. In Theorem~\ref{thm1}, we give the number of caterpillar trees using the function from~\cite{jas1}, defined as 
$c(x) = \frac{x^3 (1 - 3x^2)}{(1 - 2x)(1 - 2x^2)}$. 

In Theorem~\ref{paa}, we present $\sigma_{t}(G)$ for any graph $G$ of order $n \geq 3$:

\begin{proposition}\label{pr01}
	Let $G(V,E)$ be a graph of order $n \geq 0$. The relationship between the Albertson index and sigma index is given by:
	\[
	\sqrt{\sigma(G)} \leq \operatorname{irr}(G) \leq \sqrt{m \sigma(G)},
	\]
	where $m = |E|$.
\end{proposition}

\begin{lemma}[\cite{13,10}]\label{lem1}
	Let $G(V,E)$ be a simple, connected graph with $|V|=n \geq 2$ and $|E|=m$. Then
\begin{equation}~\label{Preliminaries:eq1}
\operatorname{irr}(G) \leq \frac{\Delta - \delta}{\sqrt{\Delta \delta}} \sqrt{m \sum_{uv \in E(G)} d_G(u) d_G(v)}.
\end{equation}
\end{lemma}

\begin{theorem}[\cite{jas1}]\label{thm1}
	For $n \geq 3$ vertices, the number of caterpillar trees is given by the function $c(x)$ using $p = n + 4$:
\begin{equation}~\label{Preliminaries:eq2}
C_n = 2^{n - 4} + 2^{\lfloor n/2 - 2 \rfloor}.
\end{equation}
\end{theorem}

We denote by $C(n, m)$ a caterpillar tree with $(n,m)$ vertices, referring to $n$ main vertices (see Figure~\ref{fig1}), where $X=(x_1, x_2, \dots, x_n)$ with $n = |X|$. Let $m$ be the set of branches of main vertices.

Zhang and Wang in \cite{j1} show that “caterpillar trees” are employed in chemical graph theory to describe molecular structures. Several topological indices, including the Albertson index, have been used to study molecular characteristics.

Figure~\ref{fig1} shows caterpillar trees of the type $C(n,3)$.

\begin{figure}[H]
	\centering
	\begin{tikzpicture}[scale=.9]
		\draw (2,2) -- (4,2);
		\draw (4,2) -- (6,2);
		\draw [line width=2pt,dash pattern=on 1pt off 1pt] (6,2) -- (8,2);
		\draw (2,2) -- (1.35,1);
		\draw (2,2) -- (2,1);
		\draw (2,2) -- (2.77,1.02);
		\draw (4,2) -- (3.59,1);
		\draw (4,2) -- (4,1);
		\draw (4,2) -- (4.55,0.98);
		\draw (6,2) -- (5.59,1.06);
		\draw (6,2) -- (6,1);
		\draw (6,2) -- (6.53,1.04);
		\draw (8,2) -- (7.55,1.02);
		\draw (8,2) -- (8,1);
		\draw (8,2) -- (8.59,1.02);
		\draw (3,2) -- (2,3);
		\draw (3,2) -- (3,3);
		\draw (3,2) -- (4,3);
		\draw (1.58,2.36) node[anchor=north west] {$x_1$};
		\draw (2.88,2.01) node[anchor=north west] {$x_2$};
		\draw (3.87,2.54) node[anchor=north west] {$x_3$};
		\draw (4.87,2.02) node[anchor=north west] {$x_4$};
		\draw (7.88,2.48) node[anchor=north west] {$x_n$};
		\draw (5,2) -- (4.33,3.05);
		\draw (5,2) -- (5,3);
		\draw (5,2) -- (5.66,3.03);
		\draw (5.84,2.57) node[anchor=north west] {$x_5$};
		\begin{scriptsize}
			\foreach \x/\y in {2/2,4/2,6/2,8/2,1.35/1,2/1,2.77/1.02,3.59/1,4/1,4.55/0.98,5.59/1.06,6/1,6.53/1.04,7.55/1.02,8/1,8.59/1.02,3/2,2/3,3/3,4/3,5/2,4.33/3.05,5/3,5.66/3.03} {
				\fill (\x,\y) circle (1pt);
			}
		\end{scriptsize}
	\end{tikzpicture}
	\caption{Example of Caterpillars $C(n,3)$}\label{fig1}
\end{figure}
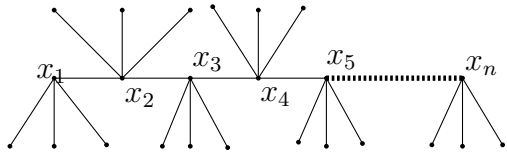

\begin{theorem}[\cite{m1}]\label{paa}
	Let $G$ be a connected graph of order $n \in \mathbb{N}$ with $n \geq 3$. Then
	\[
	\sigma_t(G) \leq \begin{cases}
		\left\lceil \frac{n}{4} \right\rceil \cdot \left\lfloor \frac{3n}{4} \right\rfloor \left(n - 1 - \left\lceil \frac{n}{4} \right\rceil \right)^2, & \text{if } n \equiv 0 \text{ or } 3 \pmod{4}, \\
		\left\lfloor \frac{n}{4} \right\rfloor \cdot \left\lceil \frac{3n}{4} \right\rceil \left(n - 1 - \left\lfloor \frac{n}{4} \right\rfloor \right)^2, & \text{if } n \equiv 1 \text{ or } 2 \pmod{4}.
	\end{cases}
	\]
\end{theorem}

\begin{definition}[Chemical Molecules~\cite{GrossYellen18}]
 The benzene molecule shown in Figure~\ref{figmmocn1} has double bonds for some pairs of its atoms, so it is modeled by a non-simple
graph.
\end{definition}
Since each carbon atom has valence 4, corresponding to four electrons in its
outer shell, it is represented by a vertex of degree 4; and since each hydrogen atom has one electron in its only shell, it is represented by a vertex of degree 1.
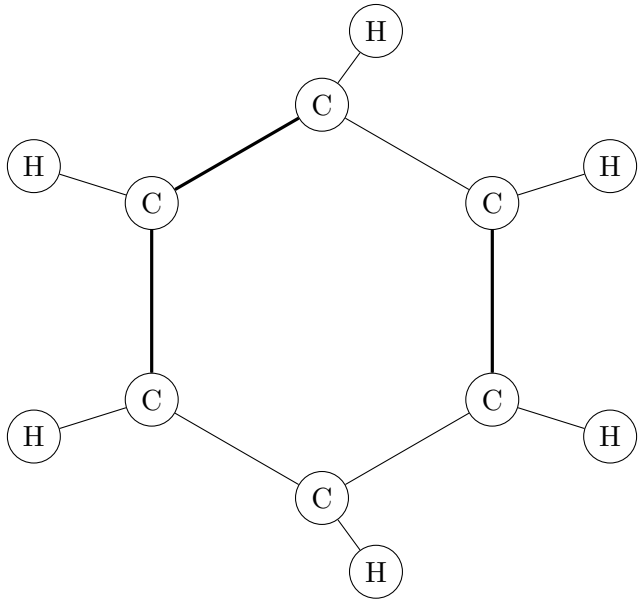
\begin{figure}[H]
    \centering
   \begin{tikzpicture}[scale=1.3, every node/.style={circle,draw,minimum size=7mm,inner sep=1pt, font=\small}]
\node (C1) at (90:2cm) {C};
\node (C2) at (30:2cm) {C};
\node (C3) at (-30:2cm) {C};
\node (C4) at (-90:2cm) {C};
\node (C5) at (-150:2cm) {C};
\node (C6) at (150:2cm) {C};
\node[draw=none, minimum size=0pt] (H1pos) at (90:2.75cm) {};
\node[draw=none, minimum size=0pt] (H2pos) at (30:2.75cm) {};
\node[draw=none, minimum size=0pt] (H3pos) at (-30:2.75cm) {};
\node[draw=none, minimum size=0pt] (H4pos) at (-90:2.75cm) {};
\node[draw=none, minimum size=0pt] (H5pos) at (-150:2.75cm) {};
\node[draw=none, minimum size=0pt] (H6pos) at (150:2.75cm) {};

\node[right=3mm of H1pos, circle,draw, minimum size=7mm] (H1) {H};
\node[right=3mm of H2pos, circle,draw, minimum size=7mm] (H2) {H};
\node[right=3mm of H3pos, circle,draw, minimum size=7mm] (H3) {H};
\node[right=3mm of H4pos, circle,draw, minimum size=7mm] (H4) {H};
\node[left=3mm of H5pos, circle,draw, minimum size=7mm] (H5) {H};
\node[left=3mm of H6pos, circle,draw, minimum size=7mm] (H6) {H};

\draw (C1) -- (C6);
\draw (C1) -- (C2);
\draw (C3) -- (C2);
\draw (C4) -- (C3);
\draw (C5) -- (C4);
\draw (C6) -- (C5);

\draw[line width=1.2pt] (C2) -- (C3);
\draw[line width=1.2pt] (C5) -- (C6);
\draw[line width=1.2pt] (C1) -- (C6);

\draw (C1) -- (H1);
\draw (C2) -- (H2);
\draw (C3) -- (H3);
\draw (C4) -- (H4);
\draw (C5) -- (H5);
\draw (C6) -- (H6);
\end{tikzpicture}
    \caption{The benzene molecule.}
    \label{figmmocn1}
\end{figure}

\section{Topological Indices on Caterpillar Trees}
Topological indices play a significant role in the analysis of caterpillar trees, especially due to their applications in chemical graph theory, biological modeling, and extremal graph analysis. These indices are numerical measures that capture structural properties of graphs, often correlating with physical, chemical, or biological characteristics.

In biological graph theory, topological indices like the Albertson and Sigma indices are valuable for understanding molecular topology, which influences biological activity or material stability. Caterpillar trees serve as simplified models for complex molecules or biological structures. The indices enable predicting molecular behavior or physical traits based on structural features.

In subsection~\ref{sec:Albertson}, we presented Albertson index on caterpillar trees. This results are employed in subsection~\ref{sec:Sigma} for Sigma index.  
In next subsection~\ref{subsomborn1},  we study a Sombor index and property of Sombor index such that (reduced Sombor, average Sombor).

\subsection{Albertson Index on Caterpillar Trees}~\label{sec:Albertson}

In Proposition~\ref{projas} we show caterpillar tree for sequence $\mathscr{D}=(\underbrace{n,n,\dots}_{m+1 \text{ times}},\underbrace{3,3,\dots}_{m \text{ times}})$, in Lemma~\ref{three.c} we show the Albertson index for a sequence of three vertices in many position as a special case of caterpillar tree, in Proposition~\ref{pro4} we discuse a new case of caterpillar tree of order $(n,4)$, also in Proposition~\ref{pro5} new case of caterpillar tree of order $(n,5)$, in Theorem~\ref{thm2} we show the general case of caterpillar tree of order $(n,m)$ as we show that in Figure~\ref{gene}.

\begin{proposition}\label{projas}
	For $n\geq 3$, Albertson index of Caterpillar trees given by:
\begin{equation}~\label{pro:eq1}
\operatorname{irr}(C(n,3))=\sum_{n=3} (12n-4).
\end{equation}
\end{proposition}
\begin{proof}
Let be $n,m\geq 0$ a set of integer point given as sequence $x=(x_1,x_2,\dots,x_n)$ and set $y=(y_1,y_2,\dots,y_m)$, depend on definintion of Albertson index we have $\operatorname{irr}(G) =\sum_{uv\in E(G)}^{} |d(u)-d(v)|$, then we have for vertices of order 3 the sequence $x=(x_1,x_2,x_3)$ and sequence $y=(y_1,y_2,y_3)$ defined as:  
\begin{equation}~\label{pro:eq2}
\operatorname{irr}(G)=\sum_{i=1}^{3}\sum_{j=1}^{3} |x_i-y_j|.
\end{equation}
For a sequence $x=(x_1,x_2,\dots,x_n)$ by $n$ and set $y=(y_1,y_2,\dots,y_m)$ by $m$, then we have in generally: $\operatorname{irr}(C(n,m))=(m+1)(n-2)+2$ that is hold to: 
	\begin{align*}
		\operatorname{irr}(C(3,3))  &= 2m^2+(m+1)m(n-2)+2 \\
		& = m(2m+mn-2m+n-2)+2 \\
		&=m(m+1)n-2m+2.
	\end{align*}
In Figure~\ref{gene} we show that the general case of caterpillar tree of order $(n,m)$, we use the discussion of each case separately from the other and at the end we get the sum of the results~\eqref{pro:eq2} in these cases and let start with first vertices $x_1$ we have Albertson index as: 
\begin{equation}~\label{pro:eq3}
\operatorname{irr}_{x_1}(T) =\sum_{j=1}^{m}|x_1-y_j|.
\end{equation}
Then, for $x_2$ we have according to~\eqref{pro:eq3}, 
\begin{equation}~\label{pro:eq4}
\operatorname{irr}_{x_2}(T) =\sum_{j=1}^{m}|x_2-y_j|.
\end{equation}
Thus, by considering~\eqref{pro:eq3} and \eqref{pro:eq4} for $x_i$ we find that 
\begin{equation}~\label{pro:eq5}
\operatorname{irr}(T)=\operatorname{irr}_{x_1}(T)+\operatorname{irr}_{x_2}(T)+\cdot+\operatorname{irr}_{x_i}(T).
\end{equation}

	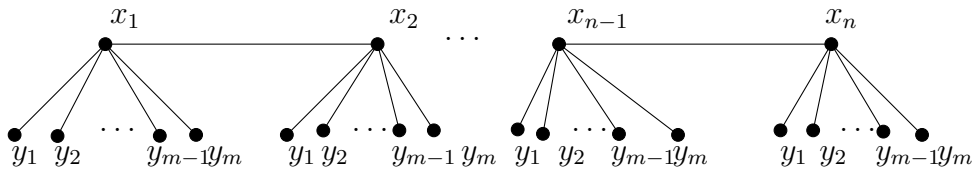
\begin{figure}[H]
		\centering
		\begin{tikzpicture}[scale=1.2]
			\draw  (3,3)-- (6,3);
			\draw  (8,3)-- (11,3);
			\draw  (3,3)-- (2,2);
			\draw  (3,3)-- (2.4730903303069725,1.9875333744351633);
			\draw  (3,3)-- (3.6,1.99);
			\draw  (3,3)-- (4,2);
			\draw  (6,3)-- (5,2);
			\draw  (6,3)-- (5.401803618428293,2.048043979561637);
			\draw  (6,3)-- (6.24,2.05);
			\draw  (6,3)-- (6.62,2.05);
			\draw  (8,3)-- (7.543879039905456,2.0601461005869317);
			\draw  (8,3)-- (7.822227823487235,2.0117376164857528);
			\draw  (8,3)-- (8.66,2.01);
			\draw  (8,3)-- (9.310788709598485,2);
			\draw  (11,3)-- (10.44,2.05);
			\draw  (11,3)-- (10.8,2.05);
			\draw  (11,3)-- (11.5738853413286,2.0359418585363422);
			\draw  (11,3)-- (12,2);
			\draw (2.9329709292681723,3.5) node[anchor=north west] {$x_1$};
			\draw (6.00690966969303,3.5) node[anchor=north west] {$x_2$};
			\draw (7.967453275790774,3.5) node[anchor=north west] {$x_{n-1}$};
			\draw (10.811451716735032,3.5) node[anchor=north west] {$x_n$};
			\draw (10.31526475469795,2) node[anchor=north west] {$y_1$};
			\draw (10.750941111608558,2) node[anchor=north west] {$y_2$};
			\draw (10.968779290063864,2.2) node[anchor=north west] {$\dots$};
			\draw (11.368149283898589,2) node[anchor=north west] {$y_{m-1}$};
			\draw (12.045868061315094,2) node[anchor=north west] {$y_m$};
			\draw (7.386551466576628,2) node[anchor=north west] {$y_1$};
			\draw (7.870636307588416,2) node[anchor=north west] {$y_2$};
			\draw (8.015861759891953,2.2) node[anchor=north west] {$\dots$};
			\draw (8.451538116802563,2) node[anchor=north west] {$y_{m-1}$};
			\draw (9.129256894219067,2) node[anchor=north west] {$y_m$};
			\draw (1.8437800369916477,2) node[anchor=north west] {$y_1$};
			\draw (2.315762756978142,2) node[anchor=north west] {$y_2$};
			\draw (2.811949719015225,2.2) node[anchor=north west] {$\dots$};
			\draw (3.3323409231028975,2) node[anchor=north west] {$y_{m-1}$};
			\draw (4,2) node[anchor=north west] {$y_m$};
			\draw (4.893514535365916,2) node[anchor=north west] {$y_1$};
			\draw (5.256578166124758,2) node[anchor=north west] {$y_2$};
			\draw (5.595437554833009,2.2) node[anchor=north west] {$\dots$};
			\draw (6.03111391174362,2) node[anchor=north west] {$y_{m-1}$};
			\draw (6.8,2) node[anchor=north west] {$y_m$};
			\draw (6.612015720957766,3.2) node[anchor=north west] {$\dots$};
			\begin{scriptsize}
				\draw [fill=black] (3,3) circle (2pt);
				\draw [fill=black] (6,3) circle (2pt);
				\draw [fill=black] (8,3) circle (2pt);
				\draw [fill=black] (11,3) circle (2pt);
				\draw [fill=black] (2,2) circle (2pt);
				\draw [fill=black] (2.4730903303069725,1.9875333744351633) circle (2pt);
				\draw [fill=black] (3.6,1.99) circle (2pt);
				\draw [fill=black] (4,2) circle (2pt);
				\draw [fill=black] (5,2) circle (2pt);
				\draw [fill=black] (5.401803618428293,2.048043979561637) circle (2pt);
				\draw [fill=black] (6.24,2.05) circle (2pt);
				\draw [fill=black] (6.62,2.05) circle (2pt);
				\draw [fill=black] (7.543879039905456,2.0601461005869317) circle (2pt);
				\draw [fill=black] (7.822227823487235,2.0117376164857528) circle (2pt);
				\draw [fill=black] (8.66,2.01) circle (2pt);
				\draw [fill=black] (9.310788709598485,2) circle (2pt);
				\draw [fill=black] (10.44,2.05) circle (2pt);
				\draw [fill=black] (10.8,2.05) circle (2pt);
				\draw [fill=black] (11.5738853413286,2.0359418585363422) circle (2pt);
				\draw [fill=black] (12,2) circle (2pt);
			\end{scriptsize}
		\end{tikzpicture}
		\caption{General case of Caterpillar tree.}~\label{gene}
	\end{figure}
From figure~\ref{gene} we show that clear.
	Finally, after discussing the cases, the Albertson index is the sum of all the cases discussed and is put in a general sum as: 
	\[
	\operatorname{irr}(T)=\sum_{j=1}^{m}|x_1-y_j|+\sum_{j=1}^{m}|x_2-y_j|+\dots+\sum_{j=1}^{m}|x_n-y_j|.
	\]
	Then we have: 
	\[
	\operatorname{irr}(T)=\sum_{i=1}^{n}\sum_{j=1}^{m}|x_i-y_j|.
	\] 
Finally, According to the definition of  Albertson index and \eqref{pro:eq2}--\eqref{pro:eq5} by giving some cases of vertices, we got the required result.
\end{proof}
In fact that is clear, in next table, Table\ref{tabja1} we show tested of a sequence of vertices $n=(3,4,\dots,1000)$ given as random values from which a sample was selected as: 
	\begin{table}[H]
    \centering
		\begin{tabular}{|c|c|c|c|c|c|c|c|}
			\hline
			$n$   & $\operatorname{irr}(C(n,3))$  &  $n$   & $\operatorname{irr}(C(n,3))$ &  $n$   & $\operatorname{irr}(C(n,3))$ \\ \hline
			3  & 32 & 50 & 596 & 400 & 4796 \\ \hline
			4  & 44 & 60 & 716 & 500 & 5996 \\ \hline
			5  &  56 & 70 & 836 & 600 & 7196  \\ \hline
			6   &  68 & 80 & 956 & 650 &  7796 \\ \hline
			7   &   80 & 90  & 1076  & 700  &  8396  \\ \hline
			8   &  92   &  100   & 1196  &  750 &  8996  \\ \hline
			9  &   104   &  110  & 1316  &   800 & 9596 \\ \hline
			10 & 116  &   150  &  1796   &  850 &  10196  \\ \hline
			20 & 236   &  200   &  2396   &  900   & 10796   \\ \hline
			30 &  356  &  250   &  2996   & 950   &  11396  \\ \hline
			40 & 476 &  300     &   3596 & 1000  &  11996 \\ \hline
		\end{tabular}
\caption{Some Value of Albertson index on caterpillar tree.}\label{tabja1}
	\end{table}
	
	In Figure~\ref{fig1} we show the general case of caterpillar tree of order $(n,3)$ and in Figure~\ref{fig2} we show that for many value of a sequence of vertices $n=(3,4,5,6), m=3$, Albertson index in caterpillars tree as: 
	
	\begin{figure}[H]
		\centering
		\begin{tikzpicture}[scale=.7]
			\draw  (1,5)-- (3,5);
			\draw  (3,5)-- (5,5);
			\draw  (1,5)-- (0.34,4.01);
			\draw  (1,5)-- (1,4);
			\draw  (1,5)-- (1.46,4.01);
			\draw  (3,5)-- (2.5,4.07);
			\draw  (3,5)-- (3,4);
			\draw  (3,5)-- (3.42,4.01);
			\draw  (5,5)-- (4.38,4.01);
			\draw  (5,5)-- (5,4);
			\draw  (5,5)-- (5.56,4.01);
			\draw  (7,5)-- (9,5);
			\draw  (9,5)-- (11,5);
			\draw  (11,5)-- (13,5);
			\draw  (7,5)-- (6.58,4.05);
			\draw  (7,5)-- (7,4);
			\draw  (7,5)-- (7.58,3.99);
			\draw  (9,5)-- (8.52,4.01);
			\draw  (9,5)-- (9,4);
			\draw  (9,5)-- (9.54,4.03);
			\draw  (11,5)-- (10.6,4.03);
			\draw  (11,5)-- (11,4);
			\draw  (11,5)-- (11.36,4.01);
			\draw  (13,5)-- (12.56,4.01);
			\draw  (13,5)-- (13,4);
			\draw  (13,5)-- (13.58,4.03);
			\draw  (1,2)-- (2,2);
			\draw  (2,2)-- (3,2);
			\draw  (3,2)-- (4,2);
			\draw  (4,2)-- (5,2);
			\draw  (1,2)-- (0.62,1.09);
			\draw  (1,2)-- (1,1);
			\draw  (1,2)-- (1.48,1.05);
			\draw  (2,2)-- (1.24,2.95);
			\draw  (2,2)-- (2,3);
			\draw  (2,2)-- (2.82,2.97);
			\draw  (3,2)-- (2.58,1.05);
			\draw  (3,2)-- (3,1);
			\draw  (3,2)-- (3.62,1.01);
			\draw  (4,2)-- (3.38,2.97);
			\draw  (4,2)-- (4,3);
			\draw (4,2)-- (4.64,3.01);
			\draw  (5,2)-- (4.66,1.07);
			\draw  (5,2)-- (5,1);
			\draw  (5,2)-- (5.52,1.01);
			\draw  (7,2)-- (8,2);
			\draw  (8,2)-- (9,2);
			\draw  (9,2)-- (10,2);
			\draw  (10,2)-- (11,2);
			\draw  (11,2)-- (12,2);
			\draw  (7,2)-- (6.5,1);
			\draw  (7,2)-- (7,1);
			\draw  (7,2)-- (7.5,1);
			\draw  (8,2)-- (7.5,3);
			\draw  (8,2)-- (8,3);
			\draw  (8,2)-- (8.5,3);
			\draw  (9,2)-- (8.5,1);
			\draw  (9,2)-- (9,1);
			\draw  (9,2)-- (9.5,1);
			\draw (10,2)-- (9.48,2.97);
			\draw  (10,2)-- (10,3);
			\draw  (10,2)-- (10.54,2.99);
			\draw  (11,2)-- (10.5,1);
			\draw  (11,2)-- (11,1);
			\draw  (11,2)-- (11.52,1.03);
			\draw  (12,2)-- (11.5,3);
			\draw (12,2)-- (12,3);
			\draw  (12,2)-- (12.5,3);
			\begin{scriptsize}
				\draw [fill=black] (1,5) circle (2pt);
				\draw [fill=black] (3,5) circle (2pt);
				\draw [fill=black] (5,5) circle (2pt);
				\draw [fill=black] (0.34,4.01) circle (2pt);
				\draw [fill=black] (1,4) circle (2pt);
				\draw [fill=black] (1.46,4.01) circle (2pt);
				\draw [fill=black] (2.5,4.07) circle (2pt);
				\draw [fill=black] (3,4) circle (2pt);
				\draw [fill=black] (3.42,4.01) circle (2pt);
				\draw [fill=black] (4.38,4.01) circle (2pt);
				\draw [fill=black] (5,4) circle (2pt);
				\draw [fill=black] (5.56,4.01) circle (2pt);
				\draw [fill=black] (7,5) circle (2pt);
				\draw [fill=black] (9,5) circle (2pt);
				\draw [fill=black] (11,5) circle (2pt);
				\draw [fill=black] (13,5) circle (2pt);
				\draw [fill=black] (6.58,4.05) circle (2pt);
				\draw [fill=black] (7,4) circle (2pt);
				\draw [fill=black] (7.58,3.99) circle (2pt);
				\draw [fill=black] (8.52,4.01) circle (2pt);
				\draw [fill=black] (9,4) circle (2pt);
				\draw [fill=black] (9.54,4.03) circle (2pt);
				\draw [fill=black] (10.6,4.03) circle (2pt);
				\draw [fill=black] (11,4) circle (2pt);
				\draw [fill=black] (11.36,4.01) circle (2pt);
				\draw [fill=black] (12.56,4.01) circle (2pt);
				\draw [fill=black] (13,4) circle (2pt);
				\draw [fill=black] (13.58,4.03) circle (2pt);
				\draw [fill=black] (1,2) circle (2pt);
				\draw [fill=black] (2,2) circle (2pt);
				\draw [fill=black] (3,2) circle (2pt);
				\draw [fill=black] (4,2) circle (2pt);
				\draw [fill=black] (5,2) circle (2pt);
				\draw [fill=black] (0.62,1.09) circle (2pt);
				\draw [fill=black] (1,1) circle (2pt);
				\draw [fill=black] (1.48,1.05) circle (2pt);
				\draw [fill=black] (1.24,2.95) circle (2pt);
				\draw [fill=black] (2,3) circle (2pt);
				\draw [fill=black] (2.82,2.97) circle (2pt);
				\draw [fill=black] (2.58,1.05) circle (2pt);
				\draw [fill=black] (3,1) circle (2pt);
				\draw [fill=black] (3.62,1.01) circle (2pt);
				\draw [fill=black] (3.38,2.97) circle (2pt);
				\draw [fill=black] (4,3) circle (2pt);
				\draw [fill=black] (4.64,3.01) circle (2pt);
				\draw [fill=black] (4.66,1.07) circle (2pt);
				\draw [fill=black] (5,1) circle (2pt);
				\draw [fill=black] (5.52,1.01) circle (2pt);
				\draw [fill=black] (7,2) circle (2pt);
				\draw [fill=black] (8,2) circle (2pt);
				\draw [fill=black] (9,2) circle (2pt);
				\draw [fill=black] (10,2) circle (2pt);
				\draw [fill=black] (11,2) circle (2pt);
				\draw [fill=black] (12,2) circle (2pt);
				\draw [fill=black] (6.5,1) circle (2pt);
				\draw [fill=black] (7,1) circle (2pt);
				\draw [fill=black] (7.5,1) circle (2pt);
				\draw [fill=black] (7.5,3) circle (2pt);
				\draw [fill=black] (8,3) circle (2pt);
				\draw [fill=black] (8.5,3) circle (2pt);
				\draw [fill=black] (8.5,1) circle (2pt);
				\draw [fill=black] (9,1) circle (2pt);
				\draw [fill=black] (9.5,1) circle (2pt);
				\draw [fill=black] (9.48,2.97) circle (2pt);
				\draw [fill=black] (10,3) circle (2pt);
				\draw [fill=black] (10.54,2.99) circle (2pt);
				\draw [fill=black] (10.5,1) circle (2pt);
				\draw [fill=black] (11,1) circle (2pt);
				\draw [fill=black] (11.52,1.03) circle (2pt);
				\draw [fill=black] (11.5,3) circle (2pt);
				\draw [fill=black] (12,3) circle (2pt);
				\draw [fill=black] (12.5,3) circle (2pt);
			\end{scriptsize}
		\end{tikzpicture}
		\caption{Caterpillars tree with $(n,3)$ vertices}\label{fig2}
	\end{figure}
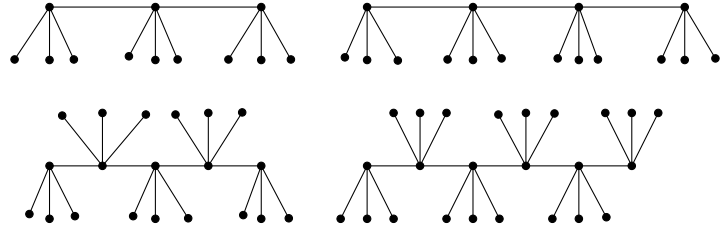
\begin{lemma}\label{three.c}
	A sequence of positive integers ${d_1},{d_2},{d_3};\quad {d_1} \ge {d_2} \ge {d_3}$  and defined numbers : $a\geq b\geq c$. so that, ``Albertson index''  defined as: 
	\[
	\operatorname{irr}=\left\lbrace
	\begin{aligned}
		& \operatorname{irr}_{max}=\sum_{a_i\in \{a,b\}}^2(a_i-1)^{2} +\sum_{i=0}^{3}k_ic^i\\
		& \operatorname{irr}_{min}= \sum_{a_i\in \{a,c\}}^2(a_i-1)^{2}+(b-1)(b-2)+(a-c)
	\end{aligned} 
	\right. 
	\]
	where \[
	k_i= \left\lbrace \begin{aligned}
		&  k_0=0 \\
		& k_1=-(ab+2a+2b) \\
		& k_2=ab+3(a+b) \\
		& k_3=-(a+b+3).
	\end{aligned}
	\right. 
	\]
\end{lemma}
\begin{proposition}~\label{maxpron1}
Let $\operatorname{irr}(C(n,3))$ be Albertson index of caterpillar tree of order $(n,3)$ vertices. Then, we have: 
\begin{equation}~\label{eq1maxpron1}
\operatorname{irr}_{\max}(C(n,3)) = \operatorname{irr}_{\min}(C(n,3)) = \operatorname{irr}(C(n,3)).
\end{equation}
\end{proposition}
This implies that we cannot specify $\operatorname{irr}_{\max}, \operatorname{irr}_{\min}$  in caterpillar tree $C(n,3)$ among Proposition~\ref{maxpron1}, or in other terms, there are no such values but only the general value of the Albertson index $\operatorname{irr}$.

\begin{proposition}\label{pro4}
Let $\operatorname{irr}(C(n,4))$ be Albertson index of caterpillar tree of order $(n,4)$ vertices. Then, we have:
\begin{equation}~\label{eq1pro4}
\operatorname{irr}(C(n,4))=\sum_{n=3}^{} 20n-6.
\end{equation}
\end{proposition}

In fact, Albertson index changes according to the selected case of caterpillar trees as in Proposition~\ref{pro4}, according to the vertices, so we presented for three vertices and then for four vertices as a discussion of this case. Furthermore, by Proposition~\ref{pro5} we emphasize that among~\eqref{eq1pro5}.
\begin{proposition}\label{pro5}
Let  $\operatorname{irr}(C(n,5))$ be Albertson index of caterpillar tree of order $(n,5)$ vertices, then we have:
\begin{equation}~\label{eq1pro5}
\operatorname{irr}(C(n,5))= \sum_{n=3}^{} 30n-8.
\end{equation}
\end{proposition}

\begin{theorem}\label{thm2}
	Let be integer number's $n ,m \geqslant 3$, and $\operatorname{irr}(C(n,m))$ Albertson index of caterpillar tree of order $(n,m)$ vertices. Then, 
\begin{equation}~\label{eq1thm2}
\operatorname{irr}(C(n,m))=\sum_{n,m\geqslant 3}(m-2)(10n-1).
\end{equation}
\end{theorem}
\begin{proof}
	Suppose the sequence of vertices $V=(v_1,v_2,\dots,v_i)$ and denote by $n=|V|$, let a sequence of branches of vertices $U=(u_1,u_2,\dots, u_j)$ and denote by $m=|U|$. For vertices $v_1$ we have: $\operatorname{irr}(C(v_1,U))=\sum_{j=3}^{p}|v_1-u_j|$ where $3\leq p \leq m$. In this case not enough because we should be take 3 vertices (see Figure~\ref{gene}) at least, so that we have for three vertices already $\operatorname{irr}(C(n,3))=\sum_{n=3}(12n-4)$. Then we suppose $X=(v_1,v_2,v_3)$ and $r=|X|$,  
		\[
		\operatorname{irr}(C(r,3))=\sum_{r=1}(12r-4).
		\]
For $n$ vertices and among all branches of vertices $U=(u_1,u_2,\dots, u_j)$ in all three case's by considering~\eqref{eq1pro4} and \eqref{eq1pro5}, in~\cite{15a} mention to all  star trees $S_{n}$ of order $n$, where Albertson index define as $\operatorname{irr}(S_{n})=(n-1)(n-2)$. Then, we have
	\[
	\zeta (m)=\sum_{m=0}(10(m-2))\quad \text{where } m=|U|.
	\]
	we have for constant terms the function $\psi(m)$ where define as sum of branches of vertices $V=(v_1,v_2,\dots,v_i)$ as : 
	\[
	\psi(m)=\sum_{m=0}\left[-(m-2)-2\right].
	\]
	Notice both functions $\zeta, \psi$ have a sequence of branches of vertices, so that we can write it in main function among all branches of vertices $U=(u_1,u_2,\dots, u_j)$ and $V=(v_1,v_2,\dots,v_i)$ as: 
\begin{equation}~\label{eq2thm2}
\Gamma(m)= \sum_{m=0}(\zeta(m)+\psi(m)).
\end{equation}
Now we can see that the function $\Gamma$ given us Albertson index for caterpillar trees, so that we can write as: 
	\begin{equation}~\label{eq3thm2}
	\operatorname{irr}(C(n,m))=\sum_{n=3}((m-2)(10n+1)-(m-2)).
\end{equation}
Finally, from~\eqref{eq2thm2} and \eqref{eq3thm2} this leads us to the relationship we assumed in the proof according to the basic definition of  Albertson index, which has at least 3 vertices, As required.
\end{proof}
\begin{theorem}\label{thm201}
	Let  $n,m,r$ be an integer numbers where $n\leq m\leq r$, and $C(m,n,r)$ is caterpillar tree with levels of $(n,m,r)$ vertices. The Albertson index of $C(m,n,r)$ where $n\geq 3,m\geq 3,r\geq 3$ given as: 
\begin{equation}~\label{eq1thm201}
\operatorname{irr}(C(n,m,r))=\sum_{n=3,r=3} \left (n^2(r^2-2r)+n(n+1) \right ).
\end{equation}
or in an equivalent variation according as:
\begin{equation}~\label{eq2thm201}
	\operatorname{irr} (C(m,n,r))=\sum\frac{13}{18}(n^3 + m^3 + r^3),
\end{equation}
	where $\{n,m,r\}\in \mathbb{N}$.
\end{theorem}
\begin{proof}
	we will discuss both cases and consider the second case as the first in the discussion and then discuss the general case.
	\begin{itemize}
		\item \textbf{Case 1:} Consider the relationship is \eqref{eq2thm201}.
	\end{itemize}
	Let be suppose the figure of caterpillar tree of order $(n,m,r)$ vertices given in figure~\ref{fig345}, and the sequences is: $V=(v_1,v_2,\dots,v_i); n=|V|$, $U=(u_1,u_2,\dots,u_j); m=|U|$ and $R=(r_1,r_2,\dots,r_k); r=|R|$, then we have as we show that in figure as: 
	
	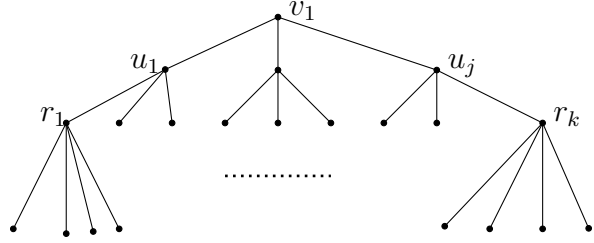
\begin{figure}[H]
		\centering
		\begin{tikzpicture}[scale=.7]
			\draw   (4,9)-- (1.87,8.01);
			\draw   (4,9)-- (4,8);
			\draw   (4,9)-- (7,8);
			\draw   (1.87,8.01)-- (0,7);
			\draw   (1.87,8.01)-- (1,7);
			\draw   (1.87,8.01)-- (2,7);
			\draw   (4,8)-- (3,7);
			\draw   (4,8)-- (4,7);
			\draw   (4,8)-- (5,7);
			\draw   (7,8)-- (6,7);
			\draw   (7,8)-- (7,7);
			\draw   (7,8)-- (9,7);
			\draw   (0,7)-- (0,4.91);
			\draw   (0,7)-- (-1,5);
			\draw   (0,7)-- (0.51,4.95);
			\draw   (0,7)-- (1,5);
			\draw   (9,7)-- (7.15,5.05);
			\draw   (9,7)-- (8,5);
			\draw   (9,7)-- (9,5);
			\draw   (9,7)-- (9.87,5.01);
			\draw (4,9.5) node[anchor=north west] {$v_1$};
			\draw (1,8.5) node[anchor=north west] {$u_1$};
			\draw (7,8.5) node[anchor=north west] {$u_j$};
			\draw (9,7.5) node[anchor=north west] {$r_k$};
			\draw (-0.7,7.5) node[anchor=north west] {$r_1$};
			\draw [line width=1pt,dotted] (3,6)-- (5,6);
			\begin{scriptsize}
				\draw [fill=black ] (4,9) circle (1.5pt);
				\draw [fill=black ] (1.87,8.01) circle (1.5pt);
				\draw [fill=black ] (4,8) circle (1.5pt);
				\draw [fill=black ] (7,8) circle (1.5pt);
				\draw [fill=black ] (0,7) circle (1.5pt);
				\draw [fill=black ] (1,7) circle (1.5pt);
				\draw [fill=black ] (2,7) circle (1.5pt);
				\draw [fill=black ] (3,7) circle (1.5pt);
				\draw [fill=black ] (4,7) circle (1.5pt);
				\draw [fill=black ] (5,7) circle (1.5pt);
				\draw [fill=black ] (6,7) circle (1.5pt);
				\draw [fill=black ] (7,7) circle (1.5pt);
				\draw [fill=black ] (9,7) circle (1.5pt);
				\draw [fill=black] (0,4.91) circle (1.5pt);
				\draw [fill=black ] (-1,5) circle (1.5pt);
				\draw [fill=black ] (0.51,4.95) circle (1.5pt);
				\draw [fill=black ] (1,5) circle (1.5pt);
				\draw [fill=black ] (7.15,5.05) circle (1.5pt);
				\draw [fill=black ] (8,5) circle (1.5pt);
				\draw [fill=black ] (9,5) circle (1.5pt);
				\draw [fill=black ] (9.87,5.01) circle (1.5pt);
			\end{scriptsize}
		\end{tikzpicture}
		\caption{Caterpillar tree of order $(3,4,5)$}
		\label{fig345}
	\end{figure}
	
	the condition of $(n,m,r)$ will be $n\geq 3, m\geq (n+1), r\geq (m+1)$, so that for vertices $v_1$, we notice that for caterpillar tree of order $(3,4,5)$, we suppose the function $\zeta(a,b,c)$ where define as $\zeta(a,b,c)=a^3+b^3+c^3$, then for $(3,4,5)$ we have $\zeta(3,4,5)=216$, by using constant $\mathcal{K}=\frac{13}{18}$, in this case we have Albertson index of caterpillar tree $C(3,4,5)$ is: $\operatorname{irr}(C(3,4,5))=\mathcal{K}\zeta(3,4,5)$, then we have: $\operatorname{irr}(C(3,4,5))=156$.
	Here in this case we have obtained the function $\zeta(a,b,c)$ and the constant $\mathcal{K}$ associated with this function, and from it in the general case if we change the values, which obtain similar results that we can express through the following table~\ref{tabcond} that shows these values according to the constant that we considered as $\mathcal{K}$ as follows:
\begin{table}[H]
		\centering
		\begin{tabular}{|c|c|c|c|}
			\hline 
			Conditions &  $(a,b,c)$ & $\zeta(a,b,c)$ & $\mathcal{K} (\zeta(a,b,c))$   \\ \hline
			$n\geq 3, m\geq (n+1), r\geq (m+1)$ &    $(3,4,5)$ & 216 & 156 \\ \hline
			$n\geq 3, m\geq (n+1), r\geq (m+2)$&     $(3, 4, 6)$ & 307 & $3991/18$ \\ \hline
			$n\geq 3, m\geq (n+2), r\geq (m+1)$&      $(3, 5, 6)$ & 368 & $4784/18$ \\ \hline
			$n\geq 3, m\geq (n+1), r\geq (m+1)$ &        $(4,5,6)$ & 405 & 420 \\ \hline
			$n\geq 4, m\geq (n+1), r\geq (m+2)$     &   $(4, 5, 7)$ & 532 & $6916$ \\ \hline
			$n\geq 4, m\geq (n+2), r\geq (m+1)$     & $(4, 6, 7)$ & 623 & $8099/18$ \\ \hline
			$n\geq 3, m\geq (n+1), r\geq (m+1)$  &      $(5,6,7)$ &  648 & 930 \\ \hline
			$n\geq 3, m\geq (n+1), r\geq (m+1)$  &      $(6,7,8)$  &  1071  & 1806 \\ \hline
		\end{tabular}
		\caption{Value of $(m,n,r)$ with determination condition.}
		\label{tabcond}
\end{table}
	From the table, we can see that the resulting values are different if the initial condition is changed, so we do not get normal values, which is why we emphasise that the condition $n\geq 3, m\geq (n+1), r\geq (m+1)$ must remain consecutive for both $(n,m,r)$.
	\begin{itemize}
		\item \textbf{Case 2:} General formula~\eqref{eq1thm201}.
	\end{itemize}
	Assume a sequence is $v=(a,b,c)$ where $a\geq 3, b=a+1,c=b+1$, or in another formula we can be write $a\geq 3, b=a+1, c=a+2$, then we have, for case of vertices $b,c$ we have: 
\begin{equation}~\label{eq3thm201}
\operatorname{irr}(C(b,c))=c^2b - 2cb + 2b - c^2 + 2c - 2.
\end{equation}
Then, 
\begin{equation}~\label{eq4thm201}
 \operatorname{irr}(C(b,c))=a(c^2b-2cb+2b-c^2+2c-1).   
\end{equation}
Therefore, from~\eqref{eq3thm201} and \eqref{eq4thm201} we have: 
	\begin{align*}
		\operatorname{irr}(C(a,b,c))= & a(c^2b - 2cb + 2b - c^2 + 2c - 1) \\
		&= a(c(c-2)(b-1)+2b-1)\\
		&=a(b-1)(c^2-2c)+ab.
	\end{align*}
According to our condition $a\geq 3, b=a+1, c=a+2$, we have $\operatorname{irr}(C(a,b,c))=a^2(c^2-2c)+a(a+1).$
	As desire.
\end{proof}

\begin{proposition}\label{spe1}
For among caterpillar tree $C(a,b,c,d)$, where $b\geq 3, a=b+2, c=b+1, d=a$, then Albertson index on caterpillar tree according Figure~\ref{figspecial} is: 
	\[
	\operatorname{irr}(C(a,b,c,d))=\sum_{i=0}^{6} C_i \cdot x_i
	\]
	where the sum can be written as: 
	$\sum_{i=0}^{6} C_i \cdot x_i = C_0 + C_1 a^4 + C_2 a^3 + C_3 a^2 + C_4 a + C_5 ac + C_6 c$, and $C_0 = 16, C_1 = 1, C_2 = -2, C_3 = -3, C_4 = 8, C_5 = 4
	C_6 = -4$.
\end{proposition}
\begin{proof}
Actually, this type of tree required dealing with both vertices alone, then we combine the result together, so that let be consider $C(a,b,c,d)$ caterpillar tree of order $(a,b,c,d)$ vertices, with term $b\geq 3, a=b+2, c=b+1, d=b+3$, we show that in Figure~\ref{figspecial}, then we have for vertices $b,c$ in both side as: 
	\begin{align*}
		\operatorname{irr}(C(b,c))=&4b(c-1)+4(c-b)\\
		&=4bc-8b+4c.
	\end{align*}
	for vertices $a,b,c$ we have: 
	\begin{align*}
		\operatorname{irr}(C(a,b,c)&=a\left ( b(a-1)^2+2b(a-b+1)+2(b-1)(b-2)+3(a-b) \right) \\
		&= a\left ( b(a-1)^2+2ab-7b+3a+4 \right) \\
		&=ab(a^2-6)+a(3a+4).
	\end{align*}
	
	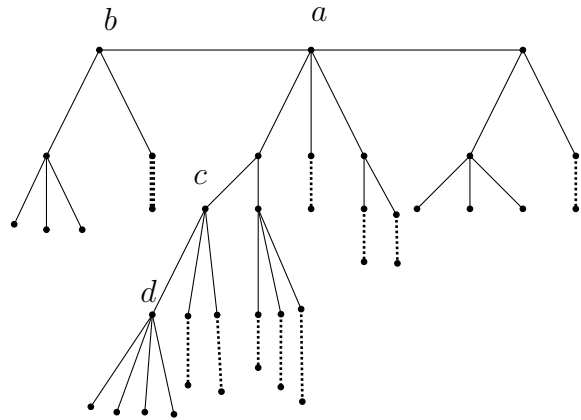
\begin{figure}[H]
		\centering
		\begin{tikzpicture}[scale=.7]
			\draw   (1,6)-- (5,6);
			\draw   (5,6)-- (9,6);
			\draw   (1,6)-- (0,4);
			\draw   (1,6)-- (2,4);
			\draw   (5,6)-- (4,4);
			\draw   (5,6)-- (6,4);
			\draw   (9,6)-- (8,4);
			\draw   (9,6)-- (10,4);
			\draw   (4,4)-- (3,3);
			\draw   (4,4)-- (4,3);
			\draw   (6,4)-- (6,3);
			\draw   (6,4)-- (6.61,2.89);
			\draw   (3,3)-- (2,1);
			\draw   (3,3)-- (2.673830436198511,0.9744263462973285);
			\draw   (3,3)-- (3.2244953014477464,0.9948213413065595);
			\draw   (4,3)-- (4,1);
			\draw   (4,3)-- (4.43,1.01);
			\draw   (4,3)-- (4.811480395927443,1.1059346675327104);
			\draw [line width=1pt,dash pattern=on 1pt off 1pt] (6,3)-- (6,2);
			\draw [line width=1pt,dash pattern=on 1pt off 1pt] (6.61,2.89)-- (6.63,1.97);
			\draw   (8,4)-- (7,3);
			\draw   (8,4)-- (9,3);
			\draw   (8,4)-- (8,3);
			\draw [line width=1pt,dash pattern=on 1pt off 1pt] (10,4)-- (10,3);
			\draw   (0,4)-- (-0.609763760287674,2.7080009220819594);
			\draw   (0,4)-- (0.6751209252938765,2.6060259470358047);
			\draw   (0,4)-- (0,2.6060259470358047);
			\draw [line width=2pt,dash pattern=on 1pt off 1pt] (2,4)-- (2,3);
			\draw   (5,6)-- (5,4);
			\draw [line width=1pt,dash pattern=on 1pt off 1pt] (5,4)-- (5,3);
			\draw (0.879070875386186,7.01134486902969) node[anchor=north west] {$b$};
			\draw (4.79490991715853,7.01134486902969) node[anchor=north west] {$a$};
			\draw (2.5718554611523556,3.9317006226358164) node[anchor=north west] {$c$};
			\draw   (2,1)-- (1.858030635829272,-0.8407282095242262);
			\draw   (2,1)-- (2.408695501078508,-0.8815181995426881);
			\draw   (2,1)-- (1.3277607655892671,-0.8407282095242262);
			\draw   (2,1)-- (0.838280885367724,-0.7387532344780714);
			\draw [line width=1pt,dash pattern=on 1pt off 1pt] (2.673830436198511,0.9744263462973285)-- (2.673830436198511,-0.3308533342934524);
			\draw [line width=1pt,dash pattern=on 1pt off 1pt] (3.2244953014477464,0.9948213413065595)-- (3.3060752814846706,-0.4532233043488381);
			\draw [line width=1pt,dash pattern=on 1pt off 1pt] (4,1)-- (4,0);
			\draw [line width=1pt,dash pattern=on 1pt off 1pt] (4.43,1.01)-- (4.4268601416128295,-0.36661144898609693);
			\draw [line width=1pt,dash pattern=on 1pt off 1pt] (4.811480395927443,1.1059346675327104)-- (4.8334586961739925,-0.6413402020679639);
			\draw (1.5725007057000386,1.8718061267034904) node[anchor=north west] {$d$};
			\begin{scriptsize}
				\draw [fill=black] (1,6) circle (1.5pt);
				\draw [fill=black] (5,6) circle (1.5pt);
				\draw [fill=black] (9,6) circle (1.5pt);
				\draw [fill=black] (0,4) circle (1.5pt);
				\draw [fill=black] (2,4) circle (1.5pt);
				\draw [fill=black] (4,4) circle (1.5pt);
				\draw [fill=black] (6,4) circle (1.5pt);
				\draw [fill=black] (8,4) circle (1.5pt);
				\draw [fill=black] (10,4) circle (1.5pt);
				\draw [fill=black] (3,3) circle (1.5pt);
				\draw [fill=black] (4,3) circle (1.5pt);
				\draw [fill=black] (6,3) circle (1.5pt);
				\draw [fill=black] (6.61,2.89) circle (1.5pt);
				\draw [fill=black] (2,1) circle (1.5pt);
				\draw [fill=black] (2.673830436198511,0.9744263462973285) circle (1.5pt);
				\draw [fill=black] (3.2244953014477464,0.9948213413065595) circle (1.5pt);
				\draw [fill=black] (4,1) circle (1.5pt);
				\draw [fill=black] (4.43,1.01) circle (1.5pt);
				\draw [fill=black] (4.811480395927443,1.1059346675327104) circle (1.5pt);
				\draw [fill=black] (6,2) circle (1.5pt);
				\draw [fill=black] (6.63,1.97) circle (1.5pt);
				\draw [fill=black] (7,3) circle (1.5pt);
				\draw [fill=black] (9,3) circle (1.5pt);
				\draw [fill=black] (8,3) circle (1.5pt);
				\draw [fill=black] (10,3) circle (1.5pt);
				\draw [fill=black] (-0.609763760287674,2.7080009220819594) circle (1.5pt);
				\draw [fill=black] (0.6751209252938765,2.6060259470358047) circle (1.5pt);
				\draw [fill=black] (0,2.6060259470358047) circle (1.5pt);
				\draw [fill=black] (2,3) circle (1.5pt);
				\draw [fill=black] (5,4) circle (1.5pt);
				\draw [fill=black] (5,3) circle (1.5pt);
				\draw [fill=black] (1.858030635829272,-0.8407282095242262) circle (1.5pt);
				\draw [fill=black] (2.408695501078508,-0.8815181995426881) circle (1.5pt);
				\draw [fill=black] (1.3277607655892671,-0.8407282095242262) circle (1.5pt);
				\draw [fill=black] (0.838280885367724,-0.7387532344780714) circle (1.5pt);
				\draw [fill=black] (2.673830436198511,-0.3308533342934524) circle (1.5pt);
				\draw [fill=black] (3.3060752814846706,-0.4532233043488381) circle (1.5pt);
				\draw [fill=black] (4,0) circle (1.5pt);
				\draw [fill=black] (4.4268601416128295,-0.36661144898609693) circle (1.5pt);
				\draw [fill=black] (4.8334586961739925,-0.6413402020679639) circle (1.5pt);
			\end{scriptsize}
		\end{tikzpicture}
		\caption{Caterpillar tree of special case in 3 vertices}
		\label{figspecial}
	\end{figure}

	So that, from all vertices, we have the formula is $a^4 - 2a^3 - 3a^2 + 8a + 4ac - 4c + 16$ , therefore we write: 
	\[
	\operatorname{irr}(C(a,b,c,d))=\sum_{i=0}^{6} C_i \cdot x_i
	\]
	Here, $x_i$ represents the variables $a, a^2, a^3, a^4, ac, c$ and the constant term, and $C_i$ are their coefficients:
	\begin{itemize}
		\item For $i = 0, x_0 = 16, C_0 = 1$, for $i = 1, x_1 = a^4, C_1 = 1$; 
		\item For $i = 2, x_2 = a^3, C_2 = -2$, for $i = 3, x_3 = a^2, C_3 = -3$; 
		\item For $i = 4, x_4 = a, C_4 = 8$, for $i = 5, x_5 = ac, C_5 = 4$ for $i = 6, x_6 = c, C_6 = -4$.
	\end{itemize}
	As desire.
\end{proof}

%----------------------------------------------------------------
%----------------------------------------------------------------
\subsection{Sigma Index on Caterpillar Trees}\label{sec:Sigma}
%----------------------------------------------------------------
%----------------------------------------------------------------

In this subsection, we introduced Sigma index of caterpillar tree is a topological index employed in graph theory for calculating the degree differences between the vertices that are adjacent, where $\sigma_{t}(G)=\sum_{1\leq i < j \leq n} (d_{i}-d_{j})^{2}$ as we mention in Theorem~\ref{paa} and Proposition~\ref{poo}, also in Proposition~\ref{sigm4} we show Sigma index on Caterpillar tree for special case of order $(n,4)$, and in Proposition~\ref{sigma5} we show that for $(n,5)$ vertices, in Theorem~\ref{thm4} we show the general case of Sigma index on caterpillar tree of order $(n,m)$ as we show that in Figure~\ref{gene}.

\begin{proposition}\label{poo}
	For $n\geq 3$, the Sigma index of the Caterpillar trees is given by:\[
	\sigma(C(n,3))=48(n-2)+56.
	\]
\end{proposition}
\begin{proof}
	Let $n,m\geq 0$ be a set of integer points given as sequence $n=(x_1,x_2,\dots,x_n)$ and set $m=(y_1,y_2,\dots,y_m)$, as we known Sigma index $\sigma$ given as: $\sigma(C(n,3))=\sum_{uv\in E(G)} (d_u-d_v)^2$, then we have: 
\begin{align*}
\sigma (C(n,3))=&(x_1-y_1)^2+(x_1-y_2)^2+(x_1-y_3)^2+(x_2-y_1)^2+\dots+\\
		&+(x_{n-1}-y_2)^2+(x_{n-1}-y_3)^2+(x_n-y_1)^2+(x_n-y_2)^2+(x_n-y_3)^2.
\end{align*}
	That is hold to Sigma index of Caterpillar tree by using $\mathcal{T}_n$ where $\mathcal{T}_n$ is a class of trees as:
    \begin{equation}~\label{eq1poo}
    \sigma(\mathcal{T}_n)=\sum_{i=1}^{n}\sum_{j=1}^{m}\left( x_i-y_j\right)^2.
    \end{equation}
	where a sequence $x=(x_1,x_2,\dots,x_n)$ and $y=(y_1,y_2,\dots,y_m)$. (see general case of caterpillar tree in Figure~\ref{gene}). In figure~\ref{cajaca1} we show that for Sigma index an Albertson indes for some value of vertices, in proposition\ref{projas} we have $\operatorname{irr}(C(n,3))=12n-4$, so that from relationship of sigma index $\sigma (C(n,3))=\sum_{i=1}^{n}\sum_{j=1}^{3}(x_i-y_j)^2$ we get on $\sigma(C(n,3))=48(n-2)+56.$ We show that in table 00 for some value of vertices $n=(3,4,\dots,100)$ as: 
	In next table (Table~\ref{tabtest}) we introduced a sample data about sigma index of caterpillar tree in rang $[3,100]$ as we show that in the table as: 
	\begin{table}[ht!]
	\centering
		\begin{tabular}{|c|c|c|c|c|c|c|c|}
			\hline
			$n$  &  $\sigma(C(n,3))$ & $n$  &  $\sigma(C(n,3))$ &  $n$  &  $\sigma(C(n,3))$ &  $n$  &  $\sigma(C(n,3))$ \\ \hline \hline
			3 & 104 & 11 & 488 & 23 & 1064 & 33 & 1544 \\ \hline
			4 & 152 & 12 & 536 & 24 & 1112 & 34 & 1592 \\ \hline
			5 & 200 & 13 & 584 & 25 & 1160 & 40 & 1880  \\ \hline
			6 & 248 & 14 & 632 & 26 & 1208 & 50 & 2360 \\ \hline
			7 & 296 & 15 & 680 & 27 & 1256 & 60 & 2840 \\ \hline
			8 & 344 & 16 & 728 & 28 & 1304 & 70  & 3320  \\ \hline
			9 & 392 & 17 & 776 & 29 & 1352 & 80 & 3800 \\ \hline
			10 & 440 & 18 & 824 & 30 & 1400 & 90 & 4280 \\ \hline
			19 & 872 & 20 & 920 & 31 & 1448 & 95 & 4520\\ \hline
			21 & 968 & 22 & 1016 & 32 & 1496 & 100 & 4760 \\ \hline
		\end{tabular}
		\caption{Sample data within the range from 3 up to 100 standard Sigma index tests}
		\label{tabtest}
	\end{table}
	
	Finally, As is well known, the Sigma index is related to the Albertson index, this condition is also given by the definition of sigma index.
\end{proof}
In next figure (Figure~\ref{cajaca1}), we show that for some value from the table~\ref{tabtest} of value Albertson index and Sigma index as: 
	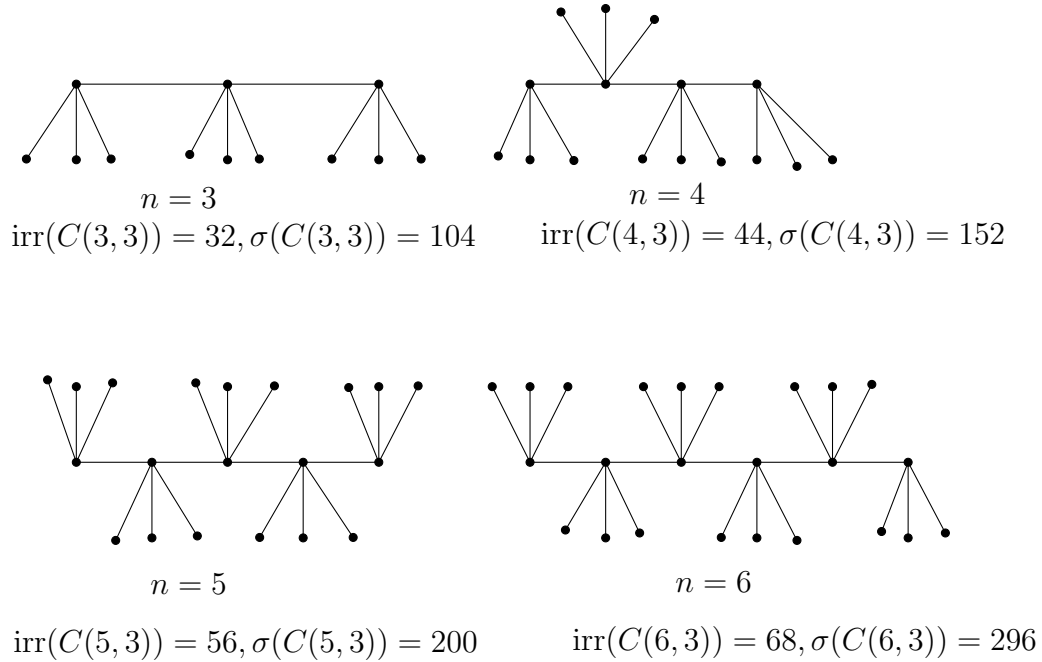
\begin{figure}[H]
    \centering
		\begin{tikzpicture}[scale=1]
			\draw  (1,5)-- (3,5);
			\draw  (3,5)-- (5,5);
			\draw  (1,5)-- (0.34,4.01);
			\draw  (1,5)-- (1,4);
			\draw  (1,5)-- (1.46,4.01);
			\draw  (3,5)-- (2.5,4.07);
			\draw  (3,5)-- (3,4);
			\draw (3,5)-- (3.42,4.01);
			\draw  (5,5)-- (4.38,4.01);
			\draw  (5,5)-- (5,4);
			\draw  (5,5)-- (5.56,4.01);
			\draw  (7,5)-- (8,5);
			\draw  (8,5)-- (9,5);
			\draw  (9,5)-- (10,5);
			\draw  (7,5)-- (6.58,4.05);
			\draw  (7,5)-- (7,4);
			\draw  (7,5)-- (7.58,3.99);
			\draw  (8,5)-- (7.409129108174678,5.954220284685677);
			\draw  (8,5)-- (8,6);
			\draw  (8,5)-- (8.645679465656283,5.856081367425232);
			\draw  (9,5)-- (8.48865719803957,4.011069722928872);
			\draw  (9,5)-- (9,4);
			\draw  (9,5)-- (9.528929721000285,3.9718141560246942);
			\draw  (10,5)-- (10,4);
			\draw (10,5)-- (10.529946677056824,3.9129308056684273);
			\draw (10,5)-- (11,4);
			\draw  (1,0)-- (2,0);
			\draw  (2,0)-- (3,0);
			\draw  (3,0)-- (4,0);
			\draw  (4,0)-- (5,0);
			\draw  (1,0)-- (0.62,1.09);
			\draw  (1,0)-- (1,1);
			\draw  (1,0)-- (1.48,1.05);
			\draw  (2,0)-- (1.5207940725479927,-1.0332706242579852);
			\draw (2,0)-- (2,-1);
			\draw  (2,0)-- (2.600322162412885,-0.9743872739017183);
			\draw  (3,0)-- (2.58,1.05);
			\draw  (3,0)-- (3,1);
			\draw  (3,0)-- (3.62,1.01);
			\draw (4,0)-- (3.424689067400621,-0.9940150573538072);
			\draw  (4,0)-- (4,-1);
			\draw  (4,0)-- (4.661239424882225,-0.9940150573538072);
			\draw  (5,0)-- (4.602356074525958,0.9883910713071755);
			\draw  (5,0)-- (5,1);
			\draw  (5,0)-- (5.52,1.01);
			\draw  (7,0)-- (8,0);
			\draw  (8,0)-- (9,0);
			\draw  (9,0)-- (10,0);
			\draw  (10,0)-- (11,0);
			\draw  (11,0)-- (12,0);
			\draw (7,0)-- (6.5,1);
			\draw  (7,0)-- (7,1);
			\draw  (7,0)-- (7.5,1);
			\draw  (8,0)-- (7.4680124585309455,-0.8958761400933626);
			\draw  (8,0)-- (8,-1);
			\draw  (8,0)-- (8.449401631135393,-0.9351317069975404);
			\draw  (9,0)-- (8.5,1);
			\draw  (9,0)-- (9,1);
			\draw  (9,0)-- (9.5,1);
			\draw  (10,0)-- (9.528929721000285,-0.9940150573538072);
			\draw  (10,0)-- (10,-1);
			\draw  (10,0)-- (10.529946677056824,-1.0332706242579852);
			\draw  (11,0)-- (10.5,1);
			\draw  (11,0)-- (11,1);
			\draw  (11,0)-- (11.52,1.03);
			\draw  (12,0)-- (11.648730333825894,-0.9155039235454515);
			\draw  (12,0)-- (12,-1);
			\draw  (12,0)-- (12.492725022265718,-0.9351317069975404);
			\draw (1.704860279845171,3.7780922646869017) node[anchor=north west] {$n=3$};
			\draw (8.166521730282554,3.8314944254343177) node[anchor=north west] {$n=4$};
			\draw (1.838365681713712,-1.321814086691353) node[anchor=north west] {$n=5$};
			\draw (8.780646578877843,-1.295113006317645) node[anchor=north west] {$n=6$};
			\draw (0,3.3241738983338633) node[anchor=north west] {$\operatorname{irr}(C(3,3))=32,\sigma(C(3,3))=104$};
			\draw (7,3.350874978707571) node[anchor=north west] {$\operatorname{irr}(C(4,3))=44,\sigma(C(4,3))=152$};
			\draw (0.02269221630155491,-2.0427432567814727) node[anchor=north west] {$\operatorname{irr}(C(5,3))=56,\sigma(C(5,3))=200$};
			\draw (7.418891479818724,-2.0160421764077645) node[anchor=north west] {$\operatorname{irr}(C(6,3))=68,\sigma(C(6,3))=296$};
			\begin{scriptsize}
				\draw [fill=black] (1,5) circle (1.5pt);
				\draw [fill=black] (3,5) circle (1.5pt);
				\draw [fill=black] (5,5) circle (1.5pt);
				\draw [fill=black] (0.34,4.01) circle (1.5pt);
				\draw [fill=black] (1,4) circle (1.5pt);
				\draw [fill=black] (1.46,4.01) circle (1.5pt);
				\draw [fill=black] (2.5,4.07) circle (1.5pt);
				\draw [fill=black] (3,4) circle (1.5pt);
				\draw [fill=black] (3.42,4.01) circle (1.5pt);
				\draw [fill=black] (4.38,4.01) circle (1.5pt);
				\draw [fill=black] (5,4) circle (1.5pt);
				\draw [fill=black] (5.56,4.01) circle (1.5pt);
				\draw [fill=black] (7,5) circle (1.5pt);
				\draw [fill=black] (8,5) circle (1.5pt);
				\draw [fill=black] (9,5) circle (1.5pt);
				\draw [fill=black] (10,5) circle (1.5pt);
				\draw [fill=black] (6.58,4.05) circle (1.5pt);
				\draw [fill=black] (7,4) circle (1.5pt);
				\draw [fill=black] (7.58,3.99) circle (1.5pt);
				\draw [fill=black] (7.409129108174678,5.954220284685677) circle (1.5pt);
				\draw [fill=black] (8,6) circle (1.5pt);
				\draw [fill=black] (8.645679465656283,5.856081367425232) circle (1.5pt);
				\draw [fill=black] (8.48865719803957,4.011069722928872) circle (1.5pt);
				\draw [fill=black] (9,4) circle (1.5pt);
				\draw [fill=black] (9.528929721000285,3.9718141560246942) circle (1.5pt);
				\draw [fill=black] (10,4) circle (1.5pt);
				\draw [fill=black] (10.529946677056824,3.9129308056684273) circle (1.5pt);
				\draw [fill=black] (11,4) circle (1.5pt);
				\draw [fill=black] (1,0) circle (1.5pt);
				\draw [fill=black] (2,0) circle (1.5pt);
				\draw [fill=black] (3,0) circle (1.5pt);
				\draw [fill=black] (4,0) circle (1.5pt);
				\draw [fill=black] (5,0) circle (1.5pt);
				\draw [fill=black] (0.62,1.09) circle (1.5pt);
				\draw [fill=black] (1,1) circle (1.5pt);
				\draw [fill=black] (1.48,1.05) circle (1.5pt);
				\draw [fill=black] (1.5207940725479927,-1.0332706242579852) circle (1.5pt);
				\draw [fill=black] (2,-1) circle (1.5pt);
				\draw [fill=black] (2.600322162412885,-0.9743872739017183) circle (1.5pt);
				\draw [fill=black] (2.58,1.05) circle (1.5pt);
				\draw [fill=black] (3,1) circle (1.5pt);
				\draw [fill=black] (3.62,1.01) circle (1.5pt);
				\draw [fill=black] (3.424689067400621,-0.9940150573538072) circle (1.5pt);
				\draw [fill=black] (4,-1) circle (1.5pt);
				\draw [fill=black] (4.661239424882225,-0.9940150573538072) circle (1.5pt);
				\draw [fill=black] (4.602356074525958,0.9883910713071755) circle (1.5pt);
				\draw [fill=black] (5,1) circle (1.5pt);
				\draw [fill=black] (5.52,1.01) circle (1.5pt);
				\draw [fill=black] (7,0) circle (1.5pt);
				\draw [fill=black] (8,0) circle (1.5pt);
				\draw [fill=black] (9,0) circle (1.5pt);
				\draw [fill=black] (10,0) circle (1.5pt);
				\draw [fill=black] (11,0) circle (1.5pt);
				\draw [fill=black] (12,0) circle (1.5pt);
				\draw [fill=black] (6.5,1) circle (1.5pt);
				\draw [fill=black] (7,1) circle (1.5pt);
				\draw [fill=black] (7.5,1) circle (1.5pt);
				\draw [fill=black] (7.4680124585309455,-0.8958761400933626) circle (1.5pt);
				\draw [fill=black] (8,-1) circle (1.5pt);
				\draw [fill=black] (8.449401631135393,-0.9351317069975404) circle (1.5pt);
				\draw [fill=black] (8.5,1) circle (1.5pt);
				\draw [fill=black] (9,1) circle (1.5pt);
				\draw [fill=black] (9.5,1) circle (1.5pt);
				\draw [fill=black] (9.528929721000285,-0.9940150573538072) circle (1.5pt);
				\draw [fill=black] (10,-1) circle (1.5pt);
				\draw [fill=black] (10.529946677056824,-1.0332706242579852) circle (1.5pt);
				\draw [fill=black] (10.5,1) circle (1.5pt);
				\draw [fill=black] (11,1) circle (1.5pt);
				\draw [fill=black] (11.52,1.03) circle (1.5pt);
				\draw [fill=black] (11.648730333825894,-0.9155039235454515) circle (1.5pt);
				\draw [fill=black] (12,-1) circle (1.5pt);
				\draw [fill=black] (12.492725022265718,-0.9351317069975404) circle (1.5pt);
			\end{scriptsize}
		\end{tikzpicture}
		\caption{Sigma index and Albertson index of some value}\label{cajaca1}
	\end{figure}

Both cases of index (Albertson, Sigma) have a different results, as we notice where $\operatorname{irr}(C(n,3))=\sum_{n=3} (12n-4)$ and $	\sigma(C(n,3))=48(n-2)+56$, in fact for make sigma index of order $(n,3)$ more effective we can computing it as: $\sigma(C(n,3))=\sum_{n=3} (48(n-2)+56)$.

\begin{proposition}
	Let be $\sigma(C(n,3))$ Sigma index of caterpillar tree of order $(n,3)$ vertices, then we have: 
	\[
	\sigma_{\max}(C(n,3)) = \sigma_{\min}(C(n,3)) = \sigma(C(n,3))
	\]
\end{proposition}
This implies that we have just $\sigma(C(n,3))$ without determining the maximum ad minimum.
\begin{proposition}\label{sigm4}
	Let be $\sigma(C(n,4))$ Sigma index of caterpillar tree of order $(n,4)$ vertices, then we have: 
	\[
	\sigma(C(n,4)) = \sum_{n=3} 100n-70.
	\]
\end{proposition}
 We show in next proposition for caterpillar tree of order $(n,5)$ vertices. 
\begin{proposition}\label{sigma5}
	Let be $\sigma(C(n,5))$ Sigma index of caterpillar tree of order $(n,5)$ vertices, then we have: 
	\[
	\sigma(C(n,5)) = \sum_{n=3} 180n-108.
	\]
\end{proposition}

From the discussions of the special cases in the propositions~\ref{poo},\ref{sigm4} and \ref{sigma5}, we can move on to the formulation of the general theorem, which gives us  Albertson index in caterpillar trees, which is the sum of the square of the difference of the base vertices and pendent vertices.

\begin{theorem}\label{thm4}
	For integer number $n\geq 3$, Sigma index for caterpillar tree $\sigma(C(n,m))$ of order $(n,m)$ vertices given by: 
	\begin{equation}~\label{eq1thm4}
	\sigma(C(n,m))=m^3n+2m^3+ m^2n-6 m^2-mn+6m-n+6.
	\end{equation}
	where $n\geq 3$ and  $m\geq n$.
\end{theorem}
\begin{proof}
We have previously discussed a variety of cases: 
	\begin{itemize}
		\item \textbf{Case 1} for $\sigma(C(n,3))$: in this case we found in proposition~\ref{poo} Sigma index on caterpillar tree of order $(n,3)$ as $\sigma(C(n,3))=\sum_{n=3} (48(n-2)+56)$.
		\item \textbf{Case 2} for $\sigma(C(n,4))$: in this case we found in proposition~\ref{sigm4} as we show that in Case 1, so that we have:  $\sigma(C(n,4)) = \sum_{n=3} 100n-70$
		\item \textbf{Case 3} as for $\sigma(C(n,5))$: in this case we found in proposition~\ref{sigma5} as $\sigma(C(n,4)) = \sum_{n=3} 180n-108$
	\end{itemize}
	Notice the coefficient of $n$ is $(48,100,108)$, then we have for a sequence of value $m$ as $\mathscr{D}=(3,4,6,7,8,\dots,20)$, but for sequence at value $m=5$ is not correct, the coefficient of $n$ with the sequence of $m$ as $\mathscr{D}=(10,20,\dots,100)$, is correct, so that in fact, these tested values are not enough to give us the full proof as we want it to be, and it must fulfil all the values without excluding any other values, so we assign the dependent an approximation of the tested values according by: 
\begin{equation}~\label{eq2thm4}
\sigma(C(n,m))\leqslant (52m-108)n-(30m+34).
\end{equation}
Consider a sequence of vertices as $\mathscr{D}=(m,\underbrace{m+2,m+2,\dots}_{n \text{ times}}, m, \underbrace{1,1,\dots}_{n+2 \text{ times}})$, then we have:
\begin{align*}
		\sigma(C)=&\sum_{i=3}\sum_{j=3} \left( d(V_i)-d(U_j) \right)^2 \\
		&2(m-1)^3+8+n(m-1)(m+1)^2\\
		&= m^3n+2m^3+ m^2n-6 m^2-mn+6m-n+6.
\end{align*}
according to~\eqref{eq2thm4}, we find that ~\eqref{eq1thm4} holds. 
\end{proof}

\begin{theorem}~\label{thm202}
Let  $n,m,r$ be an integer numbers where $n\leq m\leq r$, and $C(m,n,r)$ is caterpillar tree of order $(n,m,r)$ vertices, Sigma index of $C(m,n,r)$ where $n\geq 3,m\geq 3,r\geq 3$ given as: 
\begin{equation}~\label{eq1thm202}
\sigma(C(n,m,r))=\sum_{n,r\geqslant 3} \left (n\left(r^2(r - 3)(n + 1) + (r^3 - 3r + (3n + 1))\right) \right).
\end{equation}
or in an approximate equivalent variation according as
	\[
	\sigma (C(m,n,r))\leqslant\sum \left ( 3\left ( \frac{1}{12}(n+m+r)^3 - (n+m+r) \right ) \right ).
	\]
	where $\{n,m,r\}\in \mathbb{N}$.
\end{theorem}
\begin{proof}
	We will need to prove one of the two relationships, so we have chosen the first relationship, while the second relationship can be obtained by imposing a specific dependency and treating it as such, so for the first case we have: assume a sequence is $v=(a,b,c)$ where $a\geq 3, b=a+1,c=b+1$, or in another formula we can be write $a\geq 3, b=a+1, c=a+2$, then we have, for case of vertices $b,c$ we have: 
	\[
	((c-1)^3+1)b-1= (c^3 - 3c^2 + 3c)b + (-c^3 + 3c^2 - 3c)
	\]
	and for all vertices among sequence $v$, we have: 
	\[
	a.((c^3 - 3c^2 + 3c)b + (-c^3 + 3c^2 - 3c)+1) =a ((c^3 - 3c^2 + 3c)b + (-c^3 + 3c^2 - 3c + 1)).
	\]
	Therefore, we can write it as our condition as: 
	\begin{align*}
		\sigma(C(a,b,c))=& a ((c^3 - 3c^2 + 3c)b + (-c^3 + 3c^2 - 3c + 1))\\
		&=a((c^2(c-3)+3c)b+c^2(c-3)-3c+1)\\
		&=a(c^2(c - 3)b + (3(b - 1)c) + (c^3 - 3c + 1)).
	\end{align*}
	As we assume that $a\geq 3, b=a+1, c=a+2$. Then,  
	\begin{align*}
		\sigma(C(a,b,c)) = & a(c^2(c - 3)(a+1) + (3ac) + (c^3 - 3c + 1))\\
		&= a\left(c^2(c - 3)(a + 1) + (c^3 - 3c + (3a + 1))\right).
	\end{align*}
\end{proof}

\begin{theorem}\label{sigmathm202}
	Let be $n,m,r$ integer numbers where $n\leq m\leq r$, and $C(m,n,r)$ is caterpillar tree of order $(n,m,r)$ vertices, Sigma index of $C(m,n,r)$ where $n\geq 3,m\geq 3,r\geq 3$ given as: 
\begin{equation}~\label{eq1sigmathm202}
\sigma(C(n,m,r))=\sum_{n=3,r=3} \left (n\left(r^2(r - 3)(n + 1) + (r^3 - 3r + (3n + 1))\right) \right )
\end{equation}
	or in an approximate equivalent variation according as
	\[
	\sigma (C(m,n,r))\leqslant \sum \left ( 3\left ( \frac{1}{12}(n+m+r)^3 - (n+m+r) \right ) \right ).
	\]
	where $\{n,m,r\}\in \mathbb{N}$.
\end{theorem}
\begin{proof}
We will need to prove one of the two relationships, so we have chosen the first relationship, while the second relationship can be obtained by imposing a specific dependency and treating it as such, so for the first case we have: assume a sequence is $v=(a,b,c)$ where $a\geq 3, b=a+1,c=b+1$, or in another formula we can be write $a\geq 3, b=a+1, c=a+2$, then we have, for case of vertices $b,c$ we have: 
	\[
	((c-1)^3+1)b-1= (c^3 - 3c^2 + 3c)b + (-c^3 + 3c^2 - 3c)
	\]
	and for all vertices among sequence $v$, we have: 
	\[
	a.((c^3 - 3c^2 + 3c)b + (-c^3 + 3c^2 - 3c)+1) =a ((c^3 - 3c^2 + 3c)b + (-c^3 + 3c^2 - 3c + 1)).
	\]
	Therefore, we can write it as our condition as: 
	\begin{align*}
		\sigma(C(a,b,c))=& a ((c^3 - 3c^2 + 3c)b + (-c^3 + 3c^2 - 3c + 1))\\
		&=a((c^2(c-3)+3c)b+c^2(c-3)-3c+1)\\
		&=a(c^2(c - 3)b + (3(b - 1)c) + (c^3 - 3c + 1)).
	\end{align*}
	As we suppose $a\geq 3, b=a+1, c=a+2$, we have: 
	\begin{align*}
		\sigma(C(a,b,c)) = & a(c^2(c - 3)(a+1) + (3ac) + (c^3 - 3c + 1))\\
		&= a\left(c^2(c - 3)(a + 1) + (c^3 - 3c + (3a + 1))\right).
	\end{align*}
    As desire.
\end{proof}

\subsection{Sombor Index}~\label{subsomborn1}
Among Proposition~\ref{pro3} states a formula for the Sombor index (SO) of a caterpillar tree $\mathscr{C}(n,m)$. A caterpillar tree is a tree in which all vertices are within distance 1 of a central path, often called the ``spine.''

\begin{proposition}~\label{pro3}
For caterpillar tree $\mathscr{C}(n,m)$, we have: 
\begin{equation}~\label{eq1pro3}
\operatorname{SO}(\mathscr{C}(n,m))= 
	\begin{cases}
\sqrt{28(n-3)+66} & \quad \text{if } m=3, \\
		  \sqrt{40(n-3)+98} & \quad \text{if } m=4, \\
		\sqrt{54(n-3)+136} & \quad \text{if } m=5.
	\end{cases}
\end{equation}
\end{proposition}

The term $n-3$ from~\eqref{eq1pro3} likely reflects the number of ``internal'' vertices on the central path minus some boundary adjustments. Among Proposition~\ref{pro4} the reduced Sombor index, denoted by $\operatorname{SO}_{red}$, is a graph topological index computed by summing square roots of sums of squared degrees of adjacent vertices.

\begin{proposition}~\label{artpro4}
For caterpillar tree $\mathscr{C}(n,m)$, we have: 
\begin{equation}~\label{eq1artpro4}
\operatorname{SO}_{red}(\mathscr{C}(n,m))= 
	\begin{cases}
\sqrt{34 + (n-3)(n+1)(4)}, &  \quad \textrm{if } m=3, \\
		\sqrt{57 + 25(n-3)}, & \quad \textrm{if } m=4, \\
		\sqrt{50 + 36(n-3)}, & \quad \textrm{if } m=5.
	\end{cases}
\end{equation}
\end{proposition}

Thus, the proposition enables direct calculation of the reduced Sombor index for these families of caterpillar trees without summing over all edges explicitly. The study of Sombor index with more 2 vertices on caterpillar tree had given among Theorem\ref{thm301} and Theorem~\ref{thm302} in particular.

\begin{theorem}\label{thm301}
Let $d=(d_1,d_2,d_3)$ be a sequence where $d_1 \leq d_2 \leq d_3$, and let $\mathscr{C}(d_1,d_2,d_3)$ be a caterpillar tree of order $d$. Then, the Sombor index of the caterpillar tree is given by:
\begin{equation}~\label{eq1thm301}
\operatorname{SO}(\mathscr{C}(d_1,d_2,d_3)) = \sum_{d_1=3,d_3=4} \sqrt{d_1(2d_1^2 + 5d_1 + 7) + d_3(d_3^2 - 8d_3 + 24) - 15},
\end{equation}
where $d_1 \geq 3$, $d_2 = d_1 + 1$, $d_3 = d_1 + 2$, and $d \in \mathbb{N}$.
\end{theorem}

\begin{proof}
Let $v = (a,b,c)$ be a sequence with $a \geq 3$, $b = a + 1$, and $c = a + 2$. For the vertices $b$ and $c$, we have:
\begin{align*}
\operatorname{SO}(\mathscr{C}(a,b,c)) &= \sqrt{(c-1)(c-4)^2 + (b-1)(b+1)^2 + a(a+1)^2 + a^2 + b^2 + c^2} \\
&= \sqrt{c^3 - 8 c^2 + 24 c - 16 + b^3 + 2 b^2 - b - 1 + a^3 + 3 a^2 + a} \\
&= \sqrt{c^3 - 8 c^2 + 24 c + b^3 + 2 b^2 - b + a^3 + 3 a^2 + a - 17} \\
&= \sqrt{a^3 + b^3 + c^3 + 3 a^2 + 2 b^2 - 8 c^2 + a - b + 24 c - 17}.
\end{align*}
Substituting $b = a + 1$ and $c = a + 2$, and focusing on vertices $a$ and $c$, which represent the start and end of the caterpillar (see Figure~\ref{fig345}), we get:
\begin{align*}
\operatorname{SO}(\mathscr{C}(a,b,c)) &= \sqrt{a^3 + (a+1)^3 + c^3 + 3 a^2 + 2 (a+1)^2 - 8 c^2 + 24 c - 18} \\
&= \sqrt{2 a^3 + c^3 + 5 a^2 - 8 c^2 + 7 a + 24 c - 15} \\
&= \sqrt{a(2 a^2 + 5 a + 7) + c(c^2 - 8 c + 24) - 15}.
\end{align*}
As desire.
\end{proof}
The importance of Theorem~\ref{thm301} lies in its provision of a closed-form formula for the Sombor index of a specific family of caterpillar trees parametrized by degree sequences of three vertices. Such characterizations are foundational for further extremal graph theory studies, and for applications requiring quick calculation of topological indices, such as in cheminformatics.

\begin{theorem}\label{thm302}
Let $d = (d_1, d_2, d_3, d_4)$ be a sequence with $d_1 \leq d_2 \leq d_3 \leq d_4$, and let $\mathscr{C}(d_1, d_2, d_3, d_4)$ be the caterpillar tree of order $d$. Then, the Sombor index of the caterpillar tree is given by:
\begin{equation}~\label{eq1thm302}
\operatorname{SO}(\mathscr{C}(d_1, d_2, d_3, d_4)) = \sum_{d_i \geq 3} \sqrt{\sum_{i=1}^4 (d_i)^3 + \sum_{j \in V, i=1}^4 \alpha_j (d_i)^2 + 35 d_4 - (d_2 + d_3) + d_1 - 27},
\end{equation}
where $d_1 \geq 3$, $d_2 = d_1 + 1$, $d_3 = d_1 + 2$, $d_4 = d_1 + 3$, $i = 1, 2, \dots, k$ for some $k \in \mathbb{N}$, and $\alpha_j$ are constants for each $j \in d$ given by $\alpha_{d_1} = 3$, $\alpha_{d_2} = 2$, $\alpha_{d_3} = 2$, and $\alpha_{d_4} = -10$. 
\end{theorem}

\begin{proof}
The proof in this case relies on the general behavior of all vertices and does not split the vertices until the general formula is achieved. Let $V = (a,b,c,d)$ be a sequence where $a \geq 3$, $b = a + 1$, $c = b + 1$, and $d = c + 1$, or equivalently, $a \geq 3$, $b = a + 1$, $c = a + 2$, $d = a + 3$. We consider the vertex pairs and derive the general relationship as follows:
\begin{align*}
\operatorname{SO} &= \sqrt{(d - 1)(d - 5)^2 + d^2 + (c + 1)^2 (c - 1) + c^2} \\
&\quad + \sqrt{(b + 1)^2 (b - 1) + b^2 + a (a + 1)^2 + a^2} \\
&= \sqrt{d^3 + c^3 + b^3 + a^3 - 10 d^2 + 2 b^2 + 2 c^2 + 3 a^2 + 35 d - c - b + a - 27}.
\end{align*}

Using the assumptions $a \geq 3$, $b = a + 1$, $c = a + 2$, and $d = a + 3$, and the constants $\alpha_j$ for $j \in V$ as $\alpha_a = 3$, $\alpha_b = 2$, $\alpha_c = 2$, and $\alpha_d = -10$, we have:
\begin{equation}~\label{eq2thm302}
\operatorname{SO}(\mathscr{C}(a,b,c,d)) = \sqrt{\sum_{i=1}^4 (d_i)^3 + \sum_{j \in V, i=1}^4 \alpha_j (d_i)^2 + 35 d - (b + c) + a - 27},
\end{equation}
where each $d_i \in V$.

In this case, from~\eqref{eq2thm302} the internal sums are sufficient, and the conditions stated above must hold~\eqref{eq1thm302} for the proof to be valid.
\end{proof}
The theorem generalizes previously known results for smaller sequences, supporting broader understanding and deeper exploration of Sombor index extremal values and behaviors within families of caterpillar trees.

\subsection{Randi\'c index}
%----------------------------6--------------

\begin{proposition}\label{randpro1}
Let $R(G)$ denote the Randi\'c index, and let $\mathscr{C}(n,3)$ be a caterpillar tree with order $(n,3)$ vertices. Then, the Randi\'c index of the caterpillar tree is given by:
\begin{equation}\label{eq1randpro1}
R(\mathscr{C}(n,3)) = \sum_{n=3} \left( \frac{2n - 2}{5} + \frac{3n - 5}{5}\sqrt{5} \right),
\end{equation}
where $n \geq 3$ and $n \in \mathbb{N}$.
\end{proposition}

\begin{proof}
For a caterpillar tree of order $(n,3)$, let the vertices be $V = (v_1, v_2, \dots, v_i)$ and the subvertices be $U = (u_1, u_2, \dots, u_j)$. Each vertex in the caterpillar tree has three subvertices, so we have:
\begin{align*}
R(\mathscr{C}(n,3)) &= \frac{1}{\sqrt{d_{v_1} d_{u_1}}} + \frac{1}{\sqrt{d_{v_1} d_{u_2}}} + \frac{1}{\sqrt{d_{v_1} d_{u_3}}} + \frac{1}{\sqrt{d_{v_2} d_{u_4}}} + \\
&\quad + \frac{1}{\sqrt{d_{v_2} d_{u_5}}} + \dots + \frac{1}{\sqrt{d_{v_{i-1}} d_{u_{j-1}}}} + \frac{1}{\sqrt{d_{v_i} d_{u_j}}} \\
&= \sum_{i=1} \sum_{j=1} \frac{1}{\sqrt{d_{v_i} d_{u_j}}}.
\end{align*}
Since the number of subvertices $U$ is three for each $V$, there are several cases for the Randi\'c index based on the formula above, for example:
\begin{itemize}
\item For $n=3$, we have $R(\mathscr{C}(3,3)) = 3 + \frac{4}{5}\sqrt{5}$;
\item For $n=4$, we have $R(\mathscr{C}(4,3)) = \frac{16}{5} + \frac{7}{5}\sqrt{5}$;
\item For $n=5$, we have $R(\mathscr{C}(5,3)) = \frac{18}{5} + \frac{10}{5}\sqrt{5}$.
\end{itemize}

Now, consider the function $\zeta(a,b) = \frac{\alpha(a,b)}{5} + \frac{\beta(a,b)}{5}\sqrt{5}$. Since the number of subvertices is constant, we define the function as 
\begin{equation}~\label{eq2randpro1}
    \zeta(a,3) = \frac{\alpha(a,3)}{5} + \frac{\beta(a,3)}{5}\sqrt{5},
\end{equation}
which we write as:
\[
\zeta(a,3) = 
\begin{cases}
\frac{2a - 2}{5} + \frac{\beta(a,3)}{5}\sqrt{5}, & \text{if } n=3; \\
\frac{\alpha(a,3)}{5} + \frac{3a - 5}{5}\sqrt{5}, & \text{if } n \geq 4.
\end{cases}
\]

In all parts of the functions $\alpha$ and $\beta$, these are considered parts of the function $\zeta(a,3)$, so that we have $R(\mathscr{C}(n,3)) = \zeta(a,3).$
\end{proof}
This proposition lies in providing an explicit formula for the Randić index of a specific family of caterpillar trees. 
%----------------------------------------------------------------
\subsection{The Sum-Connectivity Index and The Harmonic Index}
%--------------------------------------------------------------
Actually, first definition of the general sum-connectivity index $\mathcal{X}_\alpha(G)$
of a graph $G$ introduced by Zhou and Trinajsti\'c in \cite{k3}, Qing Cui, Lingping Zhong in~\cite{k2} define the \emph{sum-connectivity index} as: 
\begin{equation}~\label{eq1Connectivity}
\mathcal{X}_\alpha(G)=\sum_{u v \in E(G)}(d(u)+d(v))^\alpha,
\end{equation}
Among Proposition~\ref{sumindex1}, Theorem~\ref{sumindex} the sum-connectivity index of caterpillar tree with sequence $V=(v_1,v_2,\dots,v_i)$ and sequence of sub vertices $U=(u_1,u_2,\dots,u_j)$, and the \emph{harmonic index} define as 
\begin{equation}~\label{eq2Connectivity}
H(G)=\sum_{u v \in E(G)} \frac{2}{d(u)+d(v)}.
\end{equation}

The new contribution of the Theorem~\ref{sumindex} and Proposition~\ref{sumindex1} is the derivation of explicit closed-form formulas for the sum-connectivity index of caterpillar trees $\mathscr{C}(n,m)$ with specific structures defined by sequences of vertex degrees.

\begin{proposition}\label{sumindex1}
	For caterpillar trees $\mathscr{C}(n,3)$ and $\mathscr{C}(n,4)$, the sum-connectivity index is given by: 
    \begin{equation}~\label{eq3Connectivity}
\mathcal{X}(\mathscr{C}(n,3)) = \frac{1}{\sqrt{4nd_1 - 8n - 2}}, \quad \mathcal{X}(\mathscr{C}(n,4)) = \frac{1}{\sqrt{5n(d_1 - 3) - 2}},
    \end{equation}
	where $n = |D|$ and $D$ is the sequence of vertices.
\end{proposition}

Proposition \ref{sumindex1} presents explicit formulas for the sum-connectivity index of caterpillar trees with parameters $(n,3)$ and $(n,4)$, expressing the index as functions of $n$ (the size of the vertex sequence) and $d_1$ (the smallest degree in the sequence). Such explicit closed-form equations not to be standard, where usually bounds or extremal values are studied more generally.

\begin{theorem}~\label{sumindex}
Let $D = (d_1,d_2,\dots,d_i)$ be a sequence of vertices of order $n = |D|$, where $d_1 \leq d_2 \leq \dots \leq d_i$ and define $d = \{d_1,d_2,\dots,d_i,d_1\}$. Consider $d_2 = d_3 = \dots = d_i = d_1 + 1$ and $C(n,m)$ a caterpillar tree of order $(n,m)$ vertices. Then the sum-connectivity index of the tree is given by:
\begin{equation}~\label{eq1sumindex}
\mathcal{X}(C(n,m)) = \frac{1}{\sqrt{d_1 (n m + n) - n m^2 + n - 2}}.
\end{equation}
\end{theorem}

\begin{proof}
Before proposing remedies to the issues mentioned, we discuss special cases referred to in Proposition~\ref{sumindex1} and Theorem~\ref{thm401}, and present a general formula that encompasses these results.

Assume the sequence $d = (d_1,d_2,d_3)$ of vertices on a caterpillar tree is given by three integers $a,b,c$ with $a = c$ and $b = a + 1$. Then we have:
\begin{equation}~\label{eq2sumindex}
\mathcal{X}(\mathscr{C}(3,3)) = \frac{1}{\sqrt{2a + 12(a-3) + 2b}}.
\end{equation}
Now consider the sequence $d = (d_1,d_2,d_3,d_4)$ given by integers $a,b,c,d$ such that $a = d$, $b = c$ and $b,c = a + 1$. Then:
\begin{equation}~\label{eq3sumindex}
\mathcal{X}(\mathscr{C}(4,3)) = \frac{1}{\sqrt{2a + 15(a-3) + 3b}}.
\end{equation}
For a general sequence $d = (d_1,d_2,\dots,d_i)$ where $d_2 = d_3 = \dots = d_i = d_1 + 1$, the sum from the formula $2 d_1 + 3 n (d_1 - 3) + (n-2) d_i$ can be computed as:
	\begin{align*}
		\mathcal{X}(\mathscr{C}(n,3)) &= \frac{1}{\sqrt{2 d_1 + 3 n (d_1 - 3) + (n-2) d_i}} \\
		&=  \frac{1}{\sqrt{2 d_1 + 3 n (d_1 - 3) + (n-2)(d_1 + 1)}} \\
		&= \frac{1}{\sqrt{2 d_1 + 3 n d_1 - 9 n + n d_1 + n - 2 d_1 - 2}} \\
		&= \frac{1}{\sqrt{4 n d_1 - 8 n - 2}}.
	\end{align*}
Assuming a caterpillar tree $C(n,4)$ of order $(n,4)$ vertices, with the same sequence $d = (d_1,d_2,\dots,d_i)$, then:
	\begin{align*}
		\mathcal{X}(\mathscr{C}(n,4)) &= \frac{1}{\sqrt{2 d_1 + 4 n (d_1 - 4) + (n-2) d_i}} \\
		&= \frac{1}{\sqrt{2 d_1 + 4 n (d_1 - 4) + n d_1 + n - 2 d_1 - 2}} \\
		&= \frac{1}{\sqrt{5 n (d_1 - 3) - 2}}.
	\end{align*}
Now, consider the sequence $U = (u_1,u_2,\dots,u_j)$ of sub-vertices of the sequence $D = (d_1,d_2,\dots,d_i)$, where $m = |U|$ and $\mathscr{C}(n,m)$ is a caterpillar tree of order $(n,m)$ vertices. Then:
	\begin{align*}
		\mathcal{X}(\mathscr{C}(n,m)) &= \frac{1}{\sqrt{2 d_1 + n m (d_1 - m) + (n-2) d_i}} \\
		&= \frac{1}{\sqrt{2 d_1 + n m d_1 - n m^2 + n d_1 + n - 2 d_1 - 2}} \\
		&= \frac{1}{\sqrt{d_1 (n m + n) - n m^2 + n - 2}}.
	\end{align*}
\end{proof}
Among Theorem~\ref{sumindex1}, covering specific cases $(n,3)$ and $(n,4)$ and generalizing to $(n,m)$ with a clear, simple formula.

\begin{theorem}\label{thm401}
Let $\mathscr{C}(n,m,r,p)$ be a caterpillar tree of order $(n,m,r,p)$ vertices with terms satisfying $n \leq m \leq r \leq p$, where $n \geq 3$, $m = n + 1$, $r = n + 2$, and $p = n + 3$. Then the sum-connectivity index of the tree is given by:
\begin{equation}~\label{eq1thm401}
\mathcal{X}(\mathscr{C}(n,m,r,p)) = \sum_{n \geq 3} \left(\frac{n^2}{\sqrt{1+p}} + n \left(\frac{1}{\sqrt{1 + 2n}} + \frac{1+n}{\sqrt{2 + n + p}} + \frac{2 + \sqrt{1+p}}{\sqrt{1+p}}\right) \right).
\end{equation}
This expression is written explicitly by modifying the integers $n,m,r,p$ where $i \in \{n,m,r,p\}$ and $p = n + 3$.
\end{theorem}

\begin{proof}
As in previous proofs, we consider caterpillar trees of consecutive order. We impose the condition that the number of vertices increases by 1 from one term to the next, proceeding from the end to the beginning until all vertices are considered. The sequence changes only by these increments, and any increase larger than 1 is excluded. 

Let the sequence be $d = (a,b,c)$, where $a \geq 3$, $b = a + 1$, and $c = a + 2$. Then for vertices $a$ and $b$, we have:
\begin{equation}~\label{eq2thm401}
\mathcal{X}(\mathscr{C}(a,b,c)) = a \left(\frac{a}{\sqrt{b+1}} + \frac{b}{\sqrt{b+2}} + \frac{1}{\sqrt{a+b}} + 1 \right).
\end{equation}
Now, consider the sequence $D = (a,b,c,d)$ with $a \geq 3$, $b = a + 1$, $c = b + 1$, and $d = c + 1$, which can also be written as $a \geq 3$, $b = a + 1$, $c = a + 2$, and $d = a + 3$. Then we have:
	\begin{align*}
		\mathcal{X}(\mathscr{C}(c,d)) &= a \left(\frac{c}{\sqrt{d+1}} + \frac{b}{\sqrt{d+c}} + \frac{1}{\sqrt{a+b}} + 1 \right) \\
		&= a \left(\frac{a+2}{\sqrt{d+1}} + \frac{a+1}{\sqrt{d+a+2}} + \frac{1}{\sqrt{2a+1}} + 1 \right) \\
		&= \frac{a^2}{\sqrt{1 + d}} + a \left(\frac{1}{\sqrt{1 + 2a}} + \frac{1 + a}{\sqrt{2 + a + d}} + \frac{2 + \sqrt{1 + d}}{\sqrt{1 + d}} \right).
	\end{align*}
Finally, for all vertices in the sequence $a,b,c,d$, the sum-connectivity index with $b,c$ modified is:
\begin{equation}~\label{eq3thm401}
\mathcal{X}(\mathscr{C}(a,b,c,d)) = \frac{a^2}{\sqrt{1 + d}} + a \left(\frac{1}{\sqrt{1 + 2a}} + \frac{1 + a}{\sqrt{2 + a + d}} + \frac{2 + \sqrt{1 + d}}{\sqrt{1 + d}}\right).
\end{equation}
\end{proof}
In fact, unlike the motivation in Theorem~\ref{sumindex}, the definition of the sum index does not require generalization. For instance, from the proof of Theorem~\ref{thm401}, we obtain the relation $$\mathcal{X}(\mathscr{C}(a,b,c,d)) = a \left(\frac{c}{\sqrt{d+1}} + \frac{b}{\sqrt{d+c}} + \frac{1}{\sqrt{a+b}} + 1 \right),$$ and more meaningful results follow where $$\mathcal{X}(\mathscr{C}(a,b,c,d,e)) = a \left(\frac{d}{\sqrt{e+1}} + \frac{c}{\sqrt{d+e}} + \frac{1}{\sqrt{a+b}} + 1 \right).$$

 We show that in Proposition~\ref{prosum} for a sequence of vertices of order equal degree of caterpillar tree.

\begin{proposition}\label{prosum}
Let $D = (d_1,d_2,\dots,d_n)$ be a sequence with $d_1 \leq d_2 \leq \dots \leq d_n$ and the last term satisfying $d_n = d_{n-1} - 1$. Then we have:
	\[
	\mathcal{X}(\mathscr{C}(D)) = d_1 \left(\frac{d_{n-1}}{\sqrt{d_n}} + \frac{d_{n-2}}{\sqrt{d_{n-1} + d_n}} + \frac{1}{\sqrt{d_1 + d_2}} + 1 \right).
	\]
\end{proposition}

\begin{proof}
Actually, the proof is straightforward and follows from both Theorems~\ref{sumindex} and \ref{thm401}. To avoid repetition, by using mathematical induction, we obtain the required result.
\end{proof}

%================================
\section{Open Problems}
\problem{1} For caterpillar tree $\mathscr{C}(a,b,c,d)$, find Randic index with the term $R(\mathscr{C}(a,b,c,d))$.
Through the previous discussion, it became clear to us some basic concepts that will help us find the solution to this problem.
\problem{2} Let $R(\mathscr{C}(a,b,c,d))$ be the randic index of caterpillar tree. Then, find the maximum value of $R(\mathscr{C}(a,b,c,d))$ and the minimum value of $R(\mathscr{C}(a,b,c,d))$. 
In this problem, the Albertson index provides an excellent introduction to finding the maximum and minimum values of this index.

Let $\mathscr{D}=(d_1,d_2,\cdots,d_n)$ be a degree sequence. Then, the maximum value of $R(\mathscr{C}(a,b,c,d))$ satisfied with $d_n<d_1<\dots<d_2<d_{n-1}$. Also, the minimum value of $R(\mathscr{C}(a,b,c,d))$ satisfied with $d_n\geqslant d_{n-1}\geqslant \cdots \geqslant d_1$.

\problem{3} For Sombor index $\operatorname{SO}(T)$ with $\mathscr{D}=(d_1,d_2,\cdots,d_n)$ be a degree sequence. Then, the maximum value of $\operatorname{SO}(T)$ satisfied with $d_n<d_1<\dots<d_2<d_{n-1}$. Also, the minimum value of $\operatorname{SO}(T)$ satisfied with $d_n\geqslant d_{n-1}\geqslant \cdots \geqslant d_1$.
In this problem, through discussing the results related to the Sigma index, especially paper~\cite{HamouDuaaPn2T}, in which we discussed the outlier values of this index, it achieves the intended purpose of this relation.

\section{Conclusion}
%====================
Through this paper, there are many varied topological indices. To study them on caterpillar trees, we need a large number of research papers that cover both general and special cases. In this paper, we present a study of two topological indices, the Albertson index and the Sigma index, on caterpillar trees in the general case indicated in Figure~\ref{gene}. As mentioned in the introduction, these indices have important applications in chemistry and computer science. We have established the main result for triple vertices, where $\operatorname{irr}(C(n,3)) = 12n - 4$, and for the Sigma index on caterpillar trees we have $\sigma(C(n,3)) = 48(n-2) + 56$. We also show that, in both cases, the maximum and minimum occur in a single case. Finally, we obtained a set of characteristics discussing the special cases, which we then included in the general case.

%===========================
\section*{Declarations}
\begin{itemize}
	\item Funding: Not Funding.
	\item Conflict of interest/Competing interests: The author declare that there are no conflicts of interest or competing interests related to this study.
	\item Ethics approval and consent to participate: The author contributed equally to this work.
	\item Data availability statement: All data is included within the manuscript.
\end{itemize}

\end{document}